\documentclass[11pt,reqno]{amsart}

\usepackage{color}
\usepackage[colorlinks=true,allcolors=blue,backref=page]{hyperref}
\usepackage{amsmath, amssymb, amsthm}
\usepackage{mathrsfs}
\usepackage{mathtools}
\usepackage[noabbrev,capitalize,nameinlink]{cleveref}
\crefname{equation}{}{}
\usepackage{fullpage}
\usepackage[noadjust]{cite}
\usepackage{graphics}
\usepackage{pifont}
\usepackage{tikz}
\usetikzlibrary{external}
\usepackage{bbm}
\usepackage[T1]{fontenc}

\usetikzlibrary{arrows.meta}

\usepackage{environ}
\usepackage{framed}
\usepackage{url}
\usepackage[linesnumbered,ruled,vlined]{algorithm2e}
\usepackage[noend]{algpseudocode}
\usepackage[labelfont=bf]{caption}
\usepackage{cite}
\usepackage{framed}
\usepackage[framemethod=tikz]{mdframed}
\usepackage{appendix}
\usepackage{graphicx}
\usepackage[textsize=tiny]{todonotes}
\usepackage{tcolorbox}
\usepackage{enumerate}
\usepackage{verbatim}
\allowdisplaybreaks[1]

\apptocmd{\sloppy}{\hbadness 10000\relax}{}{} % magic BibTeX spacing fix

\crefname{algocf}{Algorithm}{Algorithms}

\crefname{equation}{}{} %remove ``Equation''
\crefname{conjecture}{Conjecture}{Conjectures} %add ``Conjecture''
\AtBeginEnvironment{appendices}{\crefalias{section}{appendix}} %appendices

\usepackage[color,final]{showkeys} %add in 'final' into parameter to remove showkeys

\colorlet{refkey}{orange!20}
\colorlet{labelkey}{blue!30}

\crefname{algocf}{Algorithm}{Algorithms}

\let\originalleft\left
\let\originalright\right
\renewcommand{\left}{\mathopen{}\mathclose\bgroup\originalleft}
\renewcommand{\right}{\aftergroup\egroup\originalright}

\numberwithin{equation}{section}
\newtheorem{theorem}{Theorem}[section]
\newtheorem{proposition}[theorem]{Proposition}
\newtheorem{lemma}[theorem]{Lemma}

\crefname{claim}{Claim}{Claims}

\newtheorem*{question*}{Question}

\theoremstyle{definition}
\newtheorem{definition}[theorem]{Definition}

\newtheorem*{definition*}{Definition}

\theoremstyle{remark}
\newtheorem{remark}[theorem]{Remark}

\newcommand{\floor}[1]{\left\lfloor #1 \right\rfloor}
\newcommand{\ceil}[1]{\left\lceil #1 \right\rceil}

\newcommand{\mb}{\mathbb}
\newcommand{\mbf}{\mathbf}
\newcommand{\mbm}{\mathbbm}
\newcommand{\mc}{\mathcal}

\newcommand{\mr}{\mathrm}

\newcommand{\ol}{\overline}
\newcommand{\on}{\operatorname}

\allowdisplaybreaks

\title{Singularity of the $k$-Core of a Random Graph}

\author[Ferber]{Asaf Ferber}

\address{Department of Mathematics, University of California, Irvine.}

\email{asaff@uci.edu}

\author[Kwan]{Matthew Kwan}

\address{Department of Mathematics, Stanford University, Stanford, CA.}

\email{mattkwan@stanford.edu}

\author[Sah]{Ashwin Sah}
\author[Sawhney]{Mehtaab Sawhney}

\address{Department of Mathematics, Massachusetts Institute of Technology, Cambridge, MA 02139, USA}
\email{\{asah,msawhney\}@mit.edu}

\thanks{Ferber was supported in part by NSF grants DMS-1954395 and DMS-1953799. Kwan was supported by NSF grant DMS-1953990. Sah and Sawhney were supported by NSF Graduate Research Fellowship Program DGE-1745302.}

\begin{document}

\global\long\def\GG{\mathbb{G}}%

\begin{abstract}
Very sparse random graphs are known to typically be singular (i.e., have singular adjacency matrix), due to the presence of ``low-degree dependencies'' such as isolated vertices and pairs of degree-1 vertices with the same neighbourhood. We prove that these kinds of dependencies are in some sense the only causes of singularity: for constants $k\ge 3$ and $\lambda > 0$, an Erd\H os--R\'enyi random graph $G\sim\GG(n,\lambda/n)$ with $n$ vertices and edge probability $\lambda/n$ typically has the property that its $k$-core (its largest subgraph with minimum degree at least $k$) is nonsingular. This resolves a conjecture of Vu from the 2014 International Congress of Mathematicians, and adds to a short list of known nonsingularity theorems for ``extremely sparse'' random matrices with density $O(1/n)$. A key aspect of our proof is a technique to extract high-degree vertices and use them to ``boost'' the rank, starting from approximate rank bounds obtainable from (non-quantitative) spectral convergence machinery due to Bordenave, Lelarge and Salez.
\end{abstract}

\maketitle

\global\long\def\E{\mathbb{E}}%
\global\long\def\NN{\mathbb{N}}%
\global\long\def\RR{\mathbb{R}}%
\global\long\def\core{\mathrm{core}}%
\global\long\def\corank{\operatorname{corank}}%
\global\long\def\rank{\operatorname{rank}}%

\section{Introduction}\label{sec:introduction}
Let $M$ be an $n\times n$ random matrix with i.i.d.\ $\on{Bernoulli}(p)$ entries (meaning that each entry $M_{ij}$ satisfies $\Pr(M_{ij}=1)=p$ and $\Pr(M_{ij}=0)=1-p$). It is a classical theorem of Koml\'os\ \cite{Kom67,Kom68} that for constant $p\in (0,1)$ and $n\to \infty$, a random Bernoulli matrix is nonsingular \emph{with high probability} (``whp'' for short): that is, $\lim_{n\to\infty}\Pr(M\text{ is singular})=0$.

A huge number of strengthenings and variations of Koml\'os' theorem have been considered over the years. Notably, Tikhomirov \cite{Tik20} showed that for constant $0 < p\le 1/2$, the singularity probability is $(1-p+o(1))^n$. In an alternative direction, Costello, Tao, and Vu~\cite{CTV06} famously proved that a random \emph{symmetric} Bernoulli matrix is also nonsingular whp. Of course, a symmetric binary matrix can be interpreted as the adjacency matrix of a graph, so the Costello--Tao--Vu theorem has an interpretation in terms of random graphs: for constant $p\in(0,1)$, an Erd\H os--R\'enyi random graph $G\sim\GG(n,p)$ has nonsingular adjacency matrix whp\footnote{There is a slight difference between a random symmetric Bernoulli matrix and the adjacency matrix of a random graph: namely, the adjacency matrix of any graph has zeroes on the diagonal. However, the same techniques usually apply to both settings, and we will not further concern ourselves with this detail.}.

If $p$ decays too rapidly with $n$ (in particular, if $p\le (1-\varepsilon)\log n/n$ for some constant $\varepsilon>0$), then for reasons related to the coupon collector problem, a typical $G\sim\GG(n,p)$ has isolated vertices, meaning that its adjacency matrix has all-zero rows and is therefore singular. Costello and Vu~\cite{CV08} proved that $\log n/n$ is in fact a \emph{sharp threshold} for singularity, in the sense that if $p\ge (1+\varepsilon)\log n/n$ (and $p$ is bounded away from $1$) then a typical $G\sim\GG(n,p)$ has nonsingular adjacency matrix. This fact was later refined and generalised by Basak and Rudelson~\cite{BR18} and Addario-Berry and Eslava~\cite{AE14}. In particular, the latter authors proved a sharp \emph{hitting time} result: if we consider the random graph \emph{process} where we start with the empty graph on $n$ vertices and add random edges one-by-one, then whp at the very same moment where the last isolated vertex disappears our graph becomes nonsingular.

Even below the threshold $\log n/n$, it is natural to ask whether the only obstacles for singularity are ``local dependencies'' such as isolated vertices. In their aforementioned paper, Costello and Vu~\cite{CV08} actually proved that for $p\ge \left(1/2+\varepsilon\right)\log n/n$, whp the subgraph obtained by deleting isolated vertices is nonsingular. In follow-up work~\cite{CV10}, they extended analysis to the regime where $p\ge c\log n/n$ for any constant $c>0$; this necessitated the consideration of more sophisticated types of local dependencies than isolated vertices. The most obvious example is \emph{cherries}: pairs of degree-$1$ vertices with the same neighbour. Another obstruction is cycles of length divisible by $4$ in which every second vertex has no neighbours outside the cycle.

A natural way to systematically eliminate these kinds of local dependencies is to iteratively delete low-degree vertices. The \emph{$k$-core} $\core_{k}(G)$ of a graph $G$ is the subgraph obtained by iteratively deleting vertices with
degree less than $k$ (in any order). Equivalently, it is the largest induced subgraph with minimum degree at least $k$. At the 2014 International Congress of Mathematicians, Vu conjectured~\cite[Conjecture~5.6]{Vu14} that for constants $k\ge 3$ and $\lambda>0$, the $k$-core of a typical $G\sim \GG(n,\lambda/n)$ is nonsingular\footnote{The $k$-core can be empty (indeed, this is typically the case if $\lambda$ is below a certain threshold, as we will soon discuss). We follow the standard convention that the empty $0\times 0$ matrix, being an identity map, is nonsingular.}. He also repeated (a substantially weakened version of) this conjecture in his recent survey on combinatorial random matrix theory~\cite{Vu20}. In this paper we prove Vu's conjecture.

\begin{theorem}\label{thm:main} 
Fix constants $k\ge 3$ and $\lambda>0$. Then the adjacency matrix of the $k$-core of $G\sim\mb{G}(n,\lambda/n)$ is nonsingular with high probability.
\end{theorem}
\begin{remark}
Vu actually conjectured that the $k$-core of the \emph{giant component} of a typical $G\sim \GG(n,\lambda/n)$ is nonsingular. This follows from \cref{thm:main}, since a graph has a singular adjacency matrix if and only if at least one of its connected components has a singular adjacency matrix. 
\end{remark}

\begin{remark}
Many random graph properties of interest are \emph{monotone}: if a graph satisfies the property, then adding edges can never destroy the property. In particular, this implies that if $G\sim\GG(n,p)$ satisfies the property whp, and $p'\ge p$, then $G'\sim\GG(n,p')$ satisfies the property whp as well. Nonsingularity is not a monotone property, so formally \cref{thm:main} leaves open the case when $p$ grows asymptotically faster than $1/n$. However, the regime where $np=O(1)$ is the most challenging, and it is not difficult to adapt our proof methods to handle the case where $np\to \infty$ (we discuss this further in \cref{rem:np-infty}). We remark that the analogue of \cref{thm:main} in the regime $np\to\infty$ was recently proved\footnote{The result in \cite{Gla21} is only stated for the 3-core, but the analysis also works for larger $k$; see \cite[Remark 1]{Gla21}.} in independent work by Glasgow~\cite{Gla21} (see also \cite{DGM21}).
\end{remark}

Before discussing the ideas in the proof of \cref{thm:main}, we provide a bit more context about $k$-cores of random graphs. The notion of the $k$-core was introduced by Bollob\'as~\cite{Bol84} in his study of the evolution of sparse random graphs.

It is not an easy task to understand the $k$-core of a random graph, but by now we have a very complete picture, at least concerning statistical aspects. In an influential paper, Pittel, Spencer and Wormald~\cite{PSW96} proved that the property of having a nonempty $k$-core has a sharp threshold: for $k\ge 3$, there is an explicit constant $\lambda_k>1$ such that if $np\le \lambda_k-\varepsilon$ (respectively, if $np\ge \lambda_k+\varepsilon$), then a typical $G\sim \GG(n,p)$ has an empty $k$-core (respectively, has a nonempty $k$-core). Moreover, they found formulas for the typical numbers of edges and vertices in the $k$-core above this critical threshold. Many alternative proofs of these results have since been discovered; see for example \cite{FR03,Coo04,Kim06,Mol05,CW06,JL07,Dar08,CCKS17,CCKS19,CCKL21,Rio08}.

Another important milestone in this area is the counterpart of \cref{thm:main} for \emph{Hamiltonicity}: Krivelevich, Lubetzky and Sudakov~\cite{KLS14} proved that for $k\ge 3$, the $k$-core of a random graph has a Hamiltonian cycle whp. There are obvious parallels between Hamiltonicity and nonsingularity: both properties are of a ``global'' nature but are trivially obstructed by the existence of certain ``local'' low-degree configurations. It is worth remarking that other than considering the $k$-core, a different way to eliminate these kinds of local obstructions is to consider \emph{random regular graphs}. That is, for some $d$, we consider a random graph drawn uniformly from the set of all $n$-vertex graphs in which every degree is exactly $d$. There are close connections between random regular graphs and the $k$-core of the Erd\H os--R\'enyi random graph (which we will discuss later in more detail), but random regular graphs have more symmetry and are therefore in many ways more tractable. It is a classical result of Robinson and Wormald~\cite{RW92} that for $d\ge 3$, random $d$-regular graphs are Hamiltonian whp, and in a recent breakthrough, M\'{e}sz\'{a}ros~\cite{Mes20} and Huang~\cite{Hua18} independently proved that for $d\ge 3$, random $d$-regular graphs are nonsingular whp. Before this, Landon, Sosoe and Yau~\cite{LSY19} proved the same for $d=d_n = n^{\Omega(1)}$, while several other authors had previously considered the simpler case of random regular \emph{directed} graphs: in particular, Cook~\cite{Cook17b} established nonsingularity for random $d$-regular directed graphs with $d\ge C (\log n)^2$ (for sufficiently large $C$), and Litvak, Lytova, Tikhomirov, Tomczak-Jaegermann and Youssef~\cite{LLTTY17} proved the same only assuming that $d\to\infty$.

To our knowledge, the aforementioned recent results due to Huang and M\'{e}sz\'{a}ros (and now \cref{thm:main}) are the only known results establishing nonsingularity for any ``extremely sparse'' model of $n\times n$ random matrices, in which the number of nonzero entries is only $O(n)$. Many of the standard arguments (for example, the ``row-by-row exposure'' ideas of Koml\'os and Costello--Tao--Vu) do not work in this regime, and Huang and M\'{e}sz\'{a}ros developed quite different enumeration-based techniques in their work. Unfortunately, there are serious obstacles towards adapting Huang and M\'{e}sz\'{a}ros' ideas to the $k$-core setting.

Instead, the starting point for our proof of \cref{thm:main} is some spectral convergence machinery of Bordenave, Lelarge and Salez~\cite{BLS11}. These authors showed that if a sequence of graphs ``locally converges'' to an appropriate Galton--Watson tree $T$, then the empirical spectral measures of the graphs converge weakly to the ``local spectral measure'' of $T$. They used this fact to obtain estimates on the rank of convergent sequences of graphs; in particular their main result was an asymptotic estimate for the rank of the adjacency matrix of a typical random graph\footnote{It is worth remarking that related rank estimates for sparse random matrices have been obtained by many authors, in part due to connections to coding theory and constraint satisfaction problems. See for example \cite{CEGHR20,CG18}, and the references within.} $G\sim \GG(n,\lambda/n)$. It turns out that the $k$-core of a typical random graph $G\sim \GG(n,\lambda/n)$ satisfies an appropriate local convergence property, and the machinery in Bordenave, Lelarge and Salez~\cite{BLS11} can be used to show that the adjacency matrix $A$ of $\core_k(G)$ is close to full rank: that is, $\corank A=o(n)$. So, our task is to ``boost'' the corank from $o(n)$ to zero. This is where our new ideas come in: we are able to ``extract'' high degree vertices from the $k$-core, in a way that does not bias too significantly the distribution of the rest of the graph, and we then add back these high-degree vertices one-by-one, typically decreasing the corank at each step. We discuss the ideas in our proof in more detail in \cref{sec:outline}.

\begin{remark}
In a previous version of this paper (which was since published at \emph{Duke Mathematical Journal}), we had remarked that (with minor modifications) our proof also gives a \emph{hitting time} result: if $k\ge 3$, then in the random graph process where we add random edges one at a time, whp the $k$-core is nonsingular at the very moment it becomes nonempty.

Since then, in discussion with Amin Coja-Oghlan, Jane Gao and Mihyun Kang we discovered that such an adaptation is actually not entirely trivial; while we believe that the methods in this paper should be applicable, the hitting-time version of our theorem should be viewed as an open problem.
%With some very minor modifications, our proof also gives a \emph{hitting time} result: if $k\ge 3$, then in the random graph process where we add random edges one at a time, whp the $k$-core is nonsingular at the very moment it becomes nonempty. We discuss this further in \cref{rmk:hitting-time}.
\end{remark}

\begin{remark}
As written, our proof does not give effective probability bounds. However, with certain modifications it is possible to prove for an explicit $C > 0$ that the conclusion of \cref{thm:main} holds with probability at least $1-(\log \log n)^C/\sqrt{\log n}$ for sufficiently large $n$ (though the cutoff implied by ``sufficiently large'' is an ineffective absolute constant). We discuss this further in \cref{rem:effective}.
\end{remark}

\begin{remark}
Our proof 
does not provide a quantitative least singular value bound. This ultimately arises from the use of the spectral convergence machinery of Bordenave, Lelarge, and Salez~\cite{BLS11} which does not obviously provide bounds on the number of singular values below (say) $1/n^{100}$. It remains a fascinating open problem to prove that whp all singular values of the $3$-core of $\mb{G}(n,\lambda/n)$ are at least $n^{-\varepsilon}$, for any constant $\varepsilon$.
\end{remark}

\begin{remark}
It may also be of interest to prove a directed analogue of \cref{thm:main}, where we define the $k$-core of a directed graph to be the subgraph obtained by iteratively deleting vertices with indegree or outdegree less than $k$. It seems plausible that our methods can be adapted to prove that for $k\ge 2$ the $k$-core of a random directed graph is nonsingular whp, though we have not chosen to pursue this direction in the present paper.
\end{remark}

\subsection*{Notation}
Our graph-theoretic notation is for the most part standard. In a graph $G$, we write $\deg_G(v)$ for the degree of a vertex $v$. Abusing notation slightly, we also write $\deg_S(v)=|N_S(v)|$ for the number of neighbours of a vertex $v$ in a vertex subset $S$. We write $G[S]$ for the subgraph of $G$ induced by the vertex subset $S$, and we write $\GG(n,p)$ for the Erd\H{o}s--R\'enyi random graph on $n$ vertices with edge density $p$. 

Our use of asymptotic notation is standard as well. For functions
$f=f\left(n\right)$ and $g=g\left(n\right)$, we write $f=O\left(g\right)$
to mean that there is a constant $C$ such that $\left|f\right|\le C\left|g\right|$,
$f=\Omega\left(g\right)$ to mean that there is a constant $c>0$
such that $f(n)\ge c\left|g(n)\right|$ for sufficiently large $n$, $f=\Theta\left(g\right)$ to mean
that that $f=O\left(g\right)$ and $f=\Omega\left(g\right)$, and $f=o\left(g\right)$ to mean that $f/g\to0$ as $n\to\infty$. We say that an event occurs with high probability (``whp'') if it holds with probability $1 - o(1)$.

\subsection*{Acknowledgements}
We thank the anonymous referees for their careful reading and for suggestions which improved the quality of the manuscript.

\section{Outline of the proof, and comparison to previous work}\label{sec:outline}

In this section we outline the ideas in our proof and compare them to previous work.

\subsection{Degree-constrained random graphs}
An important observation regarding the $k$-core of a random
graph $G\sim \GG(n,p)$ is that for any set of vertices $U$ and any degree sequence $\mbf d=(d_v)_{v\in U}$, all the graphs on the vertex set $U$ with degree sequence $(d_v)_{v\in U}$ have the same probability of occurring as $\core_k(G)$. Writing $\GG(\mbf d)$ for the uniform distribution on the set of all such graphs, it therefore would suffice to consider a set $\mathcal D$ of degree sequences such that $\core_k(G)$ has degree sequence in $\mathcal D$ whp, and to show that for any $\mbf d\in \mc D$, the degree-constrained random graph $\GG(\mbf d)$ is nonsingular whp.

So, a reasonable plan of attack would be to try to generalise the recent work of Huang~\cite{Hua18} and M\'{e}sz\'{a}ros~\cite{Mes20} on random regular graphs, to the more general setting of random graphs with a given degree sequence. Roughly speaking, the approach in these papers is to view a random matrix of interest as a random matrix over a finite field $\mb F_q$, and to study the expected size of the kernel of this matrix using careful enumerative arguments. In the case of general degree sequences, while it is possible to use similar techniques to write a (very complicated) combinatorial expression for the expected kernel size, it appears intractable to analyse the associated variational problem without the symmetry present in the $d$-regular case (which is used crucially in the work of Huang~\cite{Hua18} and  M\'{e}sz\'{a}ros~\cite{Mes20}). It is worth remarking that for certain reasonable-looking degree sequences $\mbf d$ it is simply not true that a typical $G\sim\GG(\mbf d)$ is nonsingular whp: a simple example (taken from \cite{BF11}) is the case when we have $n$ vertices, $9n/10$ of which have degree $3$ and $n/10$ of which have degree $100$. In this case, typically a majority of degree-$3$ vertices are only adjacent to degree-$100$ vertices, which implies that the adjacency matrix of $G$ has corank $\Omega(n)$ whp. We further remark that there is not necessarily a direct two-way correspondence between expected kernel size and the singularity probability: it is unclear whether carefully analysing the expected kernel size is even a viable method to establish nonsingularity for general degree sequences.

However, there are certain things we can say about quite general degree-constrained random graphs: in particular, with very mild assumptions on $\mbf d$, a typical random graph $G\sim \GG(\mbf d)$ has the property that its local statistics can be very closely approximated by an appropriate Galton--Watson tree $T$ (this is the well-known notion of local graph convergence introduced independently by Benjamini and Schramm~\cite{BS01} and by Aldous and Steele~\cite{AS04}). Using some machinery of Bordenave, Lelarge and Salez~\cite{BLS11}, one can estimate the rank of the adjacency matrix of such a graph in terms of an optimisation problem involving a probability generating function associated with $\mbf d$. The typical degree sequence of the $k$-core of a random graph was computed by Cain and Wormald~\cite{CW06} (we state their result in \cref{lem:core-degree-sequence}), so one can solve an explicit optimisation problem to see that for a typical $G\sim \GG(n,\lambda/n)$, with $A$ as the adjacency matrix of $\core_k(G)$, we have $\corank A=o(n)$ whp.

\subsection{Boosting the corank}\label{subsec:boosting}

At first glance, the above considerations may not seem particularly useful, as there is usually a big jump in difficulty between proving that a random matrix is close to full rank, and proving that it is actually of full rank. However, the $k$-core of a random graph has some properties that are very useful to us. In particular, even though the expected degrees are $(n-1)p=O(1)$, there are typically at least a few vertices in $\core_k(G)$ which have quite high degree: it is easy to show that for any function $f(n)=o(n)$, there is some $\Delta$ growing to infinity such that whp at least $f(n)$ vertices in $\core_k(G)$ have degree at least $\Delta$.

The significance of high-degree vertices is that we can use them to ``boost'' the corank, adapting the ideas of Costello, Tao and Vu~\cite{CTV06} first used to study singularity of random symmetric matrices. Indeed, we show (\cref{lem:decoupling}) that if we start with a graph satisfying a certain ``unstructured kernel property'', and we add a new vertex with a large random neighbourhood, then the corank typically decreases by 1 (unless it is already zero, in which case it typically stays at zero). This fact is proved with \emph{Littlewood--Offord}-type inequalities. In particular, we prove and apply a quadratic Littlewood--Offord inequality for random vectors uniform on a ``slice'' of the Boolean hypercube (\cref{prop:quadratic-LO-general}), which may be of independent interest.

The discussion so far suggests the following strategy. First, identify a small set $T$ of high-degree vertices in $\core_k(G)$, and let $\core_k(G)-T$ be the graph obtained by deleting the vertices in $T$ from $\core_k(G)$. Then, use the machinery of Bordenave, Lelarge and Salez to show that typically $\core_k(G)-T$ has corank at most $|T|/2$. Finally, add back the vertices in $T$ one-by-one; as long as the relevant unstructuredness property continues to hold, the corank should follow a random walk heavily biased towards zero, implying that our final graph $\core_k(G)$ has nonsingular adjacency matrix whp.

On a very high level, this is indeed more or less our strategy to prove \cref{thm:main}. For the fact that random walks biased towards zero typically end up at zero, we can adapt a simple argument of Costello, Tao and Vu (\cref{lem:random-walk}). However, the other parts of the argument are not so straightforward, and a number of additional ideas are required.

\subsection{Unstructuredness of the almost-kernel vectors}
First, we need to know that $k$-cores of random graphs (and graphs obtained from $k$-cores by deleting a few vertices) typically satisfy an appropriate ``unstructured kernel property''. Specifically, for an adjacency matrix $A$, the property we need is that vectors $\mbf v$ which are orthogonal to almost all rows of $A$ (``almost-kernel vectors'') are far from being constant vectors, in the sense that for all $x\in \RR$ we have $|\{i:v_i\ne x\}|=\Omega(n)$. (Actually, due to certain issues that arise for small $k$, the precise property we need is a bit more technical than this; see \cref{def:UKP}).

It is not a trivial matter to prove that properties of this type typically hold for degree-constrained random graphs: in the setting of random regular graphs, such a property follows from results of Backhausz and Szegedy~\cite{BS19}, and one can also deduce such a property from the estimates in Huang's paper~\cite{Hua18} (in the digraph setting, similar properties were derived in the degree-$\omega(1)$ setting by Litvak, Lytova, Tikhomirov, Tomczak-Jaegermann, and Youssef \cite{LLTTY19b}). Fortunately in our case we do not actually need to consider degree-constrained random graphs at all, and it suffices to consider properties of $G\sim\GG(n,p)$ directly. Indeed, as a consequence of some fairly simple expansion estimates, we are able to show (\cref{lem:UKP}) that whp $G$ has the property that \emph{any} induced subgraph $G[U]$ vaguely resembling a $k$-core also satisfies an appropriate unstructuredness property. For example, this holds as long as $G[U]$ has minimum degree at least $3$ and $\Omega(n)$ vertices of odd degree. Due to technical issues that arise when $k=3$, we actually also need to permit a very small number of degree-$2$ vertices.

\subsection{Efficiently extracting high-degree vertices}

The remaining issue with the strategy outlined earlier is that we need to reveal the identities of high-degree vertices in $\core_k(G)$ without revealing too much information about the rest of the graph. For example, suppose we were to na\"ively reveal the degrees of all vertices in $\core_k(G)$, define $T$ to contain the set of all vertices in $\core_k(G)$ whose degree is at least some appropriately chosen value $\Delta$, and then reveal the subgraph $\core_k(G)-T$ (in preparation for adding back the vertices of $T$, and studying the effect on the corank). Unfortunately, conditioned on information revealed so far, the neighbourhoods of vertices in $T$ are very far from being uniformly random: for example, we now know which vertices have different degrees in $\core_k(G)-T$ and in $\core_k(G)$, and these vertices must be the ones with neighbours in $T$.

We must therefore be extremely careful with the way we expose information. We fix a set $S$ of $\alpha n$ vertices, for some small $\alpha>0$, and let $T$ be the set of vertices in $S$ which have degree at least $\Delta$ (with respect to $\core_k(G)$). Now, reveal the graph $\core_k(G)-T$ (which is uniform over graphs with its degree sequence, so we can estimate its corank using the machinery of Bordenave, Lelarge and Salez; we do this in \cref{lem:rank}). Due to the deletion of the vertices in $T$, this graph $\core_k(G)-T$ will typically have a few vertices with degree less than $k$ (and we therefore know these vertices must have neighbours in $T$). However, other than this issue, by symmetry considerations the neighbourhoods of vertices in $T$ still have quite a lot of usable randomness: for example, if we fix any two sets of vertices in $\core_k(G)-S$, neither of which contain a vertex which has degree less than $k$ in $\core_k(G)-T$, then these two sets are equally likely to occur as the neighbourhood of a given vertex in $T$.

There are still certain complications remaining: for example, even though all the vertices in $T$ were chosen to have high degree in $\core_k(G)$, we cannot assume they have high degree in $\core_k(G)-S$ (if we had chosen the vertices in $T$ in that way, it would not have been true that $\core_k(G)-T$ locally converges to a Galton--Watson tree). Also, $\core_k(G)-T$ will typically have a few vertices of degree less than $2$, which means its adjacency matrix typically does not actually satisfy our unstructured kernel property. So, before starting our corank-boosting random walk, we need to make certain small adjustments to the set $T$ (being very careful to do this in a such a way that we retain useful randomness for the neighbourhoods of its vertices). The details of these adjustments are described in \cref{lem:extraction}.

\section{Anti-concentration on the Boolean slice}\label{sec:anticoncentration}
The \emph{Erd\H os--Littlewood--Offord theorem} (see~\cite[Corollary~7.8]{TV10}) is an important tool in random matrix theory. Given a sequence of nonzero real numbers $v_1,\dots,v_n$, and independent zero-one random variables $x_1,\dots,x_n\in \{0,1\}$, it gives an \emph{anti-concentration} bound for the random sum $X=x_1v_1+\dots+x_nv_n$. We will need a variant for the case where $(x_1,\dots,x_n)\in \{0,1\}^n$ is constrained to have a certain sum (such vectors are said to lie on a \emph{slice} of the Boolean hypercube).
\begin{lemma}\label{lem:linear-LO-slice}
Let $1\le d\le n/2$, $\eta>0$ and let $\mbf{v}=(v_{1},\dots,v_{n})\in \RR^n$
be a vector with no entry repeated more than $(1-\eta)n$ times. Let $\mbf{x}=(x_{1},\dots,x_{n})\in \{0,1\}^n$
be a random vector, uniformly selected from the zero-one
vectors with exactly $d$ ones, and consider any $y\in \RR$. Then
\[
\Pr(\mbf{v}^{T}\mbf{x}=y)=O((\eta d)^{-1/2}).
\]
\end{lemma}

\cref{lem:linear-LO-slice} is a direct consequence of \cite[Lemma~4.2]{FKS21}. A similar inequality for the case where $d=n/2$ was proved in \cite{LLTTY17}.

We will also need an anti-concentration inequality for \emph{quadratic} polynomials of random vectors $(x_1,\dots,x_n)\in \{0,1\}^n$ on the Boolean slice. Such an inequality does not seem to be available in the literature (though a special case appears in \cite{KST19}) so we deduce it from a quadratic inequality for the case where each $x_i$ is independent. Such an inequality was first proved by Costello, Tao and Vu~\cite{CTV06}, and the Costello--Tao--Vu inequality would already suffice for our proof of \cref{thm:main}, but the following stronger inequality is now also available.

\begin{theorem}\label{thm:Kane}
Let $\mbf{x}=(x_{1},\dots,x_{n})\in \{0,1\}^n$
be a uniformly random zero-one vector of length $n$. Let $M\in \RR^{n\times n}$ be an $n\times n$ symmetric matrix with $\Omega(n^2)$ nonzero entries, let $\mbf v\in \RR^n$ be any vector, and let $y\in \RR$ be any real number. Then 
\[
\Pr(\mbf{x}^{T}M\mbf{x}+\mbf{v}^{T}\mbf{x}=y)\le (\log n)^{O(1)}/\sqrt n.
\]
\end{theorem}

\cref{thm:Kane} is a slight improvement over a result of Meka, Nguyen and Vu~\cite{MNV16}, and may be deduced from a result of Kane~\cite{Kan14} on average sensitivity of polynomial threshold functions (see the journal version of \cite{MNV16} for this deduction).

To prove a Boolean slice version of \cref{thm:Kane} we will also need the following concentration inequality, which is a generalisation
of \cite[Lemma~2.1]{KST19}.

\begin{lemma}\label{lem:injection-concentration}
Let $m\in \mathbb{N}$, let $S$ be any finite set of size $|S| \ge m$, let $\mathcal F$ be the set of functions $\{ 1,\dots,m\} \to S$ and let $\mathcal{I}\subseteq \mathcal F$ be the set of injections
$\{ 1,\dots,m\} \to S$. Consider a function $f:\mathcal F\to\RR$
satisfying the property that  $|f(\pi)-f(\pi')|\le \sum_{i=1}^mc_{i}\mbm{1}_{\pi(i)\neq \pi'(i)}$.
Let $\pi\in\mathcal{I}$ be a uniformly random injection. Then for $t\ge 0$,
\[
\Pr\left[|f(\pi)-\E f(\pi)|\ge t\right]\le 2\exp\left(-\frac{t^{2}}{8\sum_{i=1}^{m}c_{i}^{2}}\right).
\]
\end{lemma}
\begin{proof}
We may assume without loss of generality that $c_{1}\ge\dots\ge c_{m}$. Consider the Doob martingale defined by $Z_{i}=\E[f(\pi)|\pi(1),\dots,\pi(i)]$, so in particular $Z_{0}=\E f(\pi)$ and $Z_{m}=f(\pi)$. Let
$\mathcal{L}(x_{1},\dots,x_{i})$ be the conditional distribution
of $\pi$ given $\pi(1)=x_{1},\dots,\pi(i)=x_{i}$.
We want to show that 
\[
|\E[f(\mathcal{L}(x_{1},\dots,x_{i-1},x))]-\E[f(\mathcal{L}(x_{1},\dots,x_{i-1},y))]|\le2c_{i}
\]
for any choice of distinct $x_{1},\dots,x_{i-1},x,y\in S$; this will imply that
$|Z_{i}-Z_{i-1}|$ is uniformly bounded by $2c_{i}$, so
the desired result will follow from the Azuma--Hoeffding inequality
(see for example \cite[Theorem~2.25]{JLR00}).

If $\pi$ is distributed as $\mathcal{L}(x_{1},\dots,x_{i-1},x)$,
we can change $\pi(i)$ to $y$, and if some $\pi(j)$
was already equal to $y$, change $\pi(j)$ to $x$; we
thereby obtain the distribution $\mathcal{L}(x_{1},\dots,x_{i-1},y)$.
This provides a coupling between $\mathcal{L}(x_{1},\dots,x_{i-1},x)$
and $\mathcal{L}(x_{1},\dots,x_{i-1},y)$ that differs
in only two positions $i$ and $j>i$, and since $c_{j}\le c_{i}$
this implies the required bound.
\end{proof}

Now, our Boolean slice version of \cref{thm:Kane} is as follows.
\begin{proposition}\label{prop:quadratic-LO-general}
Let $M=(m_{ij})_{i,j}$ be an $n\times n$ symmetric matrix for
which there are $\Omega(n^{4})$ different 4-tuples $(i,i',j,j')$
with $m_{ij}-m_{i'j}-m_{ij'}+m_{i'j'}\ne0$. Let $\mbf{x}=(x_{1},\dots,x_{n})$
be a random zero-one vector, uniformly selected from the zero-one
vectors with exactly $d\le n/2$ ones. Then for any vector $\mbf{v}\in\RR^{n}$
and any $x\in\RR$ we have 
\[
\Pr(\mbf{x}^{T}M\mbf{x}+\mbf{v}^{T}\mbf{x}=x)\le (\log d)^{O(1)}/\sqrt d.
\]
\end{proposition}

\begin{proof}
We need a way to generate a random set of $d$ ones in a way that
involves independent random choices. Let $\pi$ be a uniformly random
injection $\{ 1,\dots,2d\} \to\{ 1,\dots,n\} $
and let $\mbf{\xi}=(\xi_{1},\dots,\xi_{d})$ be
a sequence of independent Rademacher random variables (satisfying $\Pr(\xi_i=1)=\Pr(\xi_i=-1)=1/2$, also independent
from $\pi$). Then, we choose the positions of the $d$ ones in $\mbf{x}$
as follows. For each $i\in\{ 1,\dots,d\} $, if $\xi_{i}=1$
we set $x_{\pi(i)}=1$, and if $\xi_{i}=-1$ we set $x_{\pi(i+d)}=1$.
That is to say, first we choose $d$ random pairs of entries, and
then for each we flip a coin to decide which of the two elements in
that pair to set equal to 1.

It is routine to check (see \cite[Lemma~2.8]{KST19} or the proof of \cite[Lemma~2.3]{KS20}) that $\mbf{x}^{T}M\mbf{x}+\mbf{v}^{T}\mbf{x}$
can be expressed by a quadratic polynomial of the form
\[
f_{\pi}(\mbf{\xi})=\sum_{1\le i<j\le d}a_{ij}\xi_{i}\xi_{j}+\sum_{1\le i\le d}a_{i}\xi_{i}+a_{0},
\]
where the coefficients $a_{ij},a_{i},a_{0}$ only depend on $\pi$
(not on $\mbf{\xi}$), and in particular
\[
a_{ij}=\frac{1}{4}(m_{\pi(i),\pi(j)}-m_{\pi(i),\pi(j+m)}-m_{\pi(i+m),\pi(j)}+m_{\pi(i+m),\pi(j+m)}).
\]
Let $X$ be the number of $a_{ij}$ that are nonzero. For each $i<j$
we have $\Pr(a_{ij})=\Omega(1)$, so $\E X=\Omega(d^{2})$.
For any $i$, modifying $\pi(i)$ affects $X$ by only
$O(d)$, so by \cref{lem:injection-concentration} we have 
\[
\Pr(X\le\E X/2)\le\exp\left(-\Omega\left(\frac{(\E X/2)^{2}}{d^{2}\cdot d}\right)\right)=e^{-\Omega(d)}.
\]
For any outcome of $\pi$ satisfying $X\ge\E X/2=\Omega(d^{2})$,
by \cref{thm:Kane} we
have $\Pr(f_{\pi}(\mbf{\xi})=x)\le (\log d)^{O(1)}/\sqrt d$.
\end{proof}

\section{Large level sets of kernel vectors}\label{sec:level-sets}

In this section we prove that for a typical outcome of $G\sim \mb{G}(n,p)$, vectors which are orthogonal to almost all the rows/columns of the adjacency matrix of the $k$-core cannot be ``too structured''. This will be accomplished using very minimal information about the $k$-core and some basic expansion properties in a typical $G\sim \GG(n,p)$. Related expansion-type ideas have appeared previously in the random matrix literature (see for example \cite{Cook17b,Cook19,LLTTY17, LLTTY19}), but the details here are much simpler. First, we define our structuredness property.

\begin{definition}\label{def:UKP}
Define the support $\on{supp}(\mbf w)$ of a vector $\mbf w$ to be the set of indices $i$ such that $w_i\ne 0$. Say that an $n\times n$ symmetric matrix $A$ satisfies the \emph{unstructured kernel property} $\on{UKP}(\ell,\zeta,\eta)$ if there is a set of $\zeta n$ indices $Q\subseteq \{1,\dots,n\}$ such that every nonzero vector $\mbf v\in \mb Q^n$ with $|{\on{supp}}(A\mbf v)|\le \ell$ and $\on{supp}(A\mbf v)\cap Q=\emptyset$ satisfies $\max_{x\in \RR}|\{i:v_i=x\}|\le (1-\eta)n$.
\end{definition}

For example, an $n\times n$ matrix satisfies $\on{UKP}(0,0,\eta)$ if and only if each of its nonzero kernel vectors have level sets of size at most $(1-\eta)n$.

Now, the main result we prove in this section is as follows.

\begin{lemma}\label{lem:UKP}
For constants $\theta,\lambda>0$, there is $\eta>0$ such that the following holds. Let $G\sim\GG(n,\lambda/n)$, and say an induced subgraph $G[U]$ is \emph{good} if
\begin{itemize}
    \item it has at least $\theta n$ odd-degree vertices,
    \item it has minimum degree at least 2,
    \item there are at most $(\eta/4)n$ vertices within distance 7 of a degree-2 vertex,
    \item there is no pair of degree-2 vertices within distance 4 of each other.
\end{itemize}
Then whp $G$ is such that the adjacency matrix of every good induced subgraph satisfies the unstructured kernel property $\on{UKP}(2,\eta/3,\eta)$.
\end{lemma}

\begin{remark}
In order to prove \cref{thm:main} for $k\ge 5$, it would suffice to make the simpler definition that an induced subgraph is good if it has $\theta n$ odd-degree vertices and minimum degree at least 4. With this weaker condition, it would be easy to show that whp every good induced subgraph satisfies $\on{UKP}(2,0,\eta)$ (meaning that we do not need to consider a set of exceptional vertices $Q$ in the definition of the unstructured kernel property). Unfortunately, in the cases $k=3$ and $k=4$, with non-negligible probability there are a small number of ``bad configurations'' that force us to consider the more technical property $\on{UKP}(2,\zeta,\eta)$.
\end{remark}

\begin{remark}
Being a bit more careful with the estimates in this section, one can show that the property in \cref{lem:UKP} holds with probability $1-n^{-\Omega(1)}$.
\end{remark}

To prove \cref{lem:UKP} we will need the following basic expansion properties of $\GG(n,p)$. For vertex sets $S,T$ in a graph, let $e(S,T)$ be the number of ordered pairs $(x,y)\in S\times T$ such that $xy$ is an edge (so, edges with endpoints in $S\cap T$ are counted twice).

\begin{lemma}\label{lem:expansion-estimates-1}
For constants $\theta,\lambda>0$, there is $\eta>0$ such that the following holds. Let $G\sim \GG(n,\lambda/n)$. Then whp there is no set $S$ of fewer than $\eta n$ vertices such that at least $\theta n-2$ vertices of $G$ have a neighbour in $S$.
\end{lemma}

\begin{lemma}\label{lem:expansion-estimates-2}
For constant $\lambda>0$ there is $\eta>0$ such that the following holds. Let $G\sim \GG(n,\lambda/n)$. Then whp:
\begin{enumerate}
    \item there is no subgraph with fewer than 12 vertices and edge-density greater than 1,
    \item there are at most $\log n$ vertices adjacent to a vertex in a 4-cycle, and
    \item there is no pair of disjoint vertex sets $S,W$ satisfying the inequalities
    \[5\le |S|<\eta n,\quad e(S,S\cup W)\ge \ceil{\vphantom{x^{x^x}}5|S|/2},\quad |W|\le \ceil{\vphantom{x^{x^x}}5|S|/2}/2-e(S)+1.\]
\end{enumerate}
\end{lemma}

The proofs of the estimates in \cref{lem:expansion-estimates-1,lem:expansion-estimates-2} are quite standard, but involve some rather tedious calculations, so we first present the deduction of \cref{lem:UKP}.

\begin{proof}[Proof of \cref{lem:UKP}]
Fix an outcome of $G$ satisfying the conclusions of \cref{lem:expansion-estimates-1,lem:expansion-estimates-2} for the given constants $\theta,\lambda > 0$, choosing appropriate $\eta>0$. Consider any good $G[U]$, and write $A[U]$ for its adjacency matrix. Let $Q\subseteq U$ be the set of vertices in $G[U]$ adjacent to a vertex in a 4-cycle or within distance 7 of a degree-2 vertex (so $|Q|\le (\eta/3) n$).

We need to show that for each vector $\mbf v\in \mb Q^U$ satisfying $|{\on{supp}}(A[U]\mbf v)|\le 2$ and $\on{supp}(A[U]\mbf v)\cap Q=\emptyset$, the level sets of $\mbf v$ have size at most $(1-\eta)n$. Multiplying $\mbf v$ by a suitable integer and dividing by the greatest common divisor of its entries, it actually suffices to show the same property for vectors $\mbf{v}\in(\mb{F}_2)^U$, working mod 2.

Let $\mbf{v}\in(\mb{F}_2)^U\setminus\{\mbf{0}\}$ be such that $|{\on{supp}}(A[U]\mbf v)|\le 2$ and $\on{supp}(A[U]\mbf v)\cap Q=\emptyset$. For $x\in \mb F_2$ let $S_x=\{i\in U: v_i= x\}$. We will show that $|S_0|,|S_1|\ge \eta n$.

\textbf{The $0$-set.} First, we consider $S_0$. Every vertex $i\in U\setminus\on{supp}(A[U]\mbf v)$ has $(A[U]\mbf v)_i=0$ (mod 2), meaning that $i$ has an even number of neighbours in $S_1=U\setminus S_0$. So, all but at most two of the odd-degree vertices in $G[U]$ must have a neighbour in $S_0$. Since $G[U]$ is good, it has at least $\theta n$ odd-degree vertices so by the property in \cref{lem:expansion-estimates-1} we have $|S_0|\ge \eta n$ as desired.

\textbf{The $1$-set.} For the rest of the proof we consider $S_1$ (which is nonempty, since $\mbf v$ is nonzero). Let $F$ be an auxiliary graph on the vertex set $S_1$ with an edge $uv$ when there is a path of length at most $2$ between $u,v$ in $G[U]$. Now, consider a connected
component $F[S]$ of $F$. It suffices to prove that $|S|\ge\eta n$.

Let $s=|S|$, and let $T$ be the set of vertices in $U$ which have a neighbour in $S$. As before, every $i\in U\setminus\on{supp}(A[U]\mbf v)$ has an even number of neighbours in $S_1$, so in fact all vertices of $T\setminus \on{supp}(A[U]\mbf v)$ have at least two neighbours in $S_1$ (therefore at least two neighbours in $S$, noting that all its neighbors in $S_1$ are in $S$).

First, it is straightforward to rule out the case $s=1$: in this case, all the neighbours
of the single vertex in $S$ would be in $\on{supp}(A[U]\mbf v)$ (since
they would have odd degree into $S$ and hence $S_1$). Since $|{\on{supp}}(A[U]\mbf v)|\le2$,
this would imply the existence of a degree-2 vertex adjacent to a vertex in $\on{supp}(A[U]\mbf v)$,
which would contradict the choice of $\mbf v$ with respect to $Q$.

So, we assume $s>1$. Let $S^{*}\subseteq S$ be the set of vertices in $S$ which have degree 2 with respect to $G[U]$. In $G[U]$, there are no degree-2
vertices within distance 4 of each other, so in $F[S]$,
the closed neighbourhoods of vertices in $S^{*}$ are disjoint (in a graph, the \emph{closed neighbourhood} of a vertex $v$ is its set of neighbours together with $v$ itself). Since $F[S]$ is connected
and has more than one vertex, each vertex has at least one neighbour, so the closed neighbourhoods of vertices in $S^*$ have size at least two, implying that $|S^{*}|\le |S|/2=s/2$. 
Therefore $e(S,T)\ge\ceil{5s/2}$.
(This edge count, and every further edge count, is with respect to $G$.)

It follows that $e(S,T\backslash S)\ge\ceil{5s/2}-2e(S)$.
In case we have that $|T\backslash S|\le\ceil{(\ceil{5s/2}-2e(S)+1)/2}$
we let $W=T\backslash S$, and otherwise let $W$ be a subset of
$T\backslash S$ obtained by repeatedly deleting vertices (starting
with vertices in $\on{supp}(A[U]\mbf v)$, if any are present in $T\backslash S$)
until $\ceil{(\ceil{5s/2}-2e(S)+1)/2}$
vertices remain. In the latter case, all but at most one of the vertices
in $W$ have at least two neighbours in $S$, implying that $e(S,W)\ge\ceil{5s/2}-2e(S)$.
Either way, property (3) in \cref{lem:expansion-estimates-2} is violated, unless $2\le s\le4$ or $s\ge\eta n$.
So, it remains to rule out the cases $2\le s\le4$, with some more careful variants of the above argument.
\begin{itemize}
\item First suppose $|T\cap\on{supp}(A[U]\mbf v)|=0$. In this
case, every vertex in $T$ has at least two neighbours in $S$.
\begin{itemize}
\item For $s=2$, the sets $S$ and $T$ are disjoint, so $e(S,T)\ge\ceil{5s/2}=5$. This is only possible if $|T|\ge3$, but this would imply the existence of a complete bipartite subgraph $K_{2,3}$, violating
property (1) in \cref{lem:expansion-estimates-2}.
\item For $3\le s\le4$: starting
from $T\backslash S$, delete vertices (if necessary) to obtain a
set $W$ containing at most $\ceil{(\ceil{5s/2}-2e(S))/2}\le(\ceil{5s/2}+1)/2-e(S)$
vertices such that $e(S,W)\ge\ceil{5s/2}-2e(S)$.
Then $S\cup W$ has at least $\ceil{5s/2}-e(S)$ edges
and at most $s+(\ceil{5s/2}+1)/2-e(S)$ vertices,
violating property (1) in \cref{lem:expansion-estimates-2}.
\end{itemize}
\item Otherwise, suppose $2\le s\le4$ and $1\le|T\cap\on{supp}(A[U]\mbf v)|\le2$. Recall that $F[S]$ is connected, so all the vertices in $S$ (being within distance 7 of a vertex in $Q$) have degree
at least 3, so $e(S,T)\ge3s$.
\begin{itemize}
\item For $s=2$, there are one or two vertices in $T$ with
exactly one neighbour in $S$ (these vertices cannot be in $Q$), and to have $e(S,T)\ge 6$ there must then be
at least two vertices with exactly two neighbours in $S$. But this is impossible, because it would imply that there is a vertex not in $Q$ adjacent to a vertex in a
4-cycle.
\item For $3\le s\le4$: starting from $T\backslash S$, delete
vertices (if necessary; starting from vertices in $\on{supp}(A[U]\mbf v)$)
to obtain a set $W$ containing at most $\ceil{(3s-2e(S)+1)/2}\le(3s+2)/2-e(S)$
vertices such that $e(S,W)\ge3s-2e(S)$. Then
$S\cup W$ has at least $3s-e(S)$ edges and at most $s+(3s+2)/2-e(S)$
vertices, violating property (1) in \cref{lem:expansion-estimates-2}.
\end{itemize}
\end{itemize}
This completes the proof. 
\end{proof}

Now, we prove \cref{lem:expansion-estimates-1,lem:expansion-estimates-2}.
\begin{proof}[Proof of \cref{lem:expansion-estimates-1}]
Let $s=\floor{\eta n}$ and $f=\floor{(\theta n-2)/2}$. It suffices to show that whp there is no ``bad pair'' $(S,F)$, where $S$ is a set of $s$ vertices and $F$ is a set of $f$ edges each with at least one endpoint in $S$. Indeed, then for every set of size $s$ there are fewer than $2f\le\theta n-2$ vertices incident to $S$.

The probability a bad pair exists is at most
\[
\binom{n}{s}\binom{\binom s 2+s(n-s)}{f}p^{f}\le \left(\frac{en}s\right)^s\left(\frac{epns}f\right)^f=o(1)
\]
for sufficiently small $\eta$.
\end{proof}

\begin{proof}[Proof of \cref{lem:expansion-estimates-2}]
First, (1) and (2) trivially hold whp by Markov's inequality, observing that the expected number of subgraphs on 11 vertices with density greater than 1 is $O(n^{11}p^{12})=o(1)$, and the expected number of vertices adjacent to a 4-cycle is $O(p^5n^5)=O(1)$.

It remains to consider (3). This comes down to a union bound calculation, but we need to be rather careful with the estimates. It is convenient to handle the small $s$ and large $s$ cases separately; let $N_1=\floor{(\log n)^{1/10}}$ and $N_2=\floor{\eta n}$. First, we show that whp there is no set $S$ of $N_1\le s\le N_2$ vertices and set $F$ of $\ceil{1.1s}$ edges between vertices of $S$. Indeed, the probability that such a situation occurs is at most
\begin{align*}
    \sum_{s=N_1}^{N_2}\binom{n}{s}\binom{\binom{s}2}{\ceil{1.1s}}p^{\ceil{1.1s}}&\leq \sum_{s=N_1}^{N_2} \left(\frac{en}{s}\right)^s \left(\frac{esp}{\ceil{1.1s}}\right)^{\ceil{1.1s}}\le \sum_{s=N_1}^{N_2} \left(O(np(sp)^{0.1})\right)^s=o(1).
\end{align*}

To complete the proof that (3) holds whp, it suffices to show that whp there is no ``bad quadruple'' $(S,R,F,D)$, where:
\begin{itemize}
    \item $S,R$ are disjoint vertex sets with sizes $3\le s\le N_2$ and $r_{s,f}:=\floor{\ceil{5s/2}/2-f+1}$, respectively;
    \item $F$ is a set of $f$ edges between vertices in $S$;
    \item $D$ is a set of $\ceil{5s/2}-2f$ edges between $S$ and $R$;
    \item $f<\ceil{1.1s}$ or $s<N_1$.
\end{itemize}
The probability a bad quadruple exists is at most
\begin{equation}
\sum_{s,f}\binom{n}{s}\binom{n}{r_{s,f}}\binom{\binom s 2}{f}\binom{sr_{s,f}}{\ceil{5s/2}-2f}p^{\ceil{5s/2}-f},\label{eq:bad-quadruple}
\end{equation}
where the sum is over all $3\le s\le N_2$ and $0\le f\le \binom s 2$ satisfying $f<{1.1s}$ or $s<N_1$. We now show that this sum is $o(1)$.

First, we handle the range where $3\le s<N_1$. The contribution to \cref{eq:bad-quadruple} from these terms is at most
\begin{align*}
    \sum_{s=5}^{N_1}\sum_{f=0}^{\binom s 2}n^sn^{r_{s,f}}2^{s^2}2^{sr_{s,f}}p^{\ceil{5s/2}-f}\le \sum_{s=5}^{N_1}\sum_{f=0}^{\binom s 2}n^{o(1)}n^{-\ceil{5s/2}/2+s+1}\leq n^{-1/2+o(1)}=o(1).
\end{align*}
Second, we handle the range where $N_1\le s\le N_2$. The contribution to \cref{eq:bad-quadruple} from these terms is at most
\begin{align}
    &\sum_{s=N_1}^{N_2}\sum_{f=0}^{1.1s}\left(\frac{en}{s}\right)^s\left(\frac{en}{r_{s,f}}\right)^{r_{s,f}}\left(\frac{es^2}{f}\right)^{f}\left(\frac{sr_{s,f}}{\ceil{5s/2}-2f}\right)^{\ceil{5s/2}-2f}p^{\ceil{5s/2}-f}\notag\\
    &\qquad\leq \sum_{s=N_1}^{N_2}\sum_{f=0}^{1.1s}
    e^{O(s)} \left(\frac{n}{s}\right)^{s+\ceil{5s/2}/2-f+1}\left(\frac{s^2}{f}\right)^{f}s^{\ceil{5s/2}-2f}p^{\ceil{5s/2}-f}\notag\\
    &\qquad\le \sum_{s=N_1}^{N_2}\sum_{f=0}^{1.1s}
    e^{O(s)} \left(
    n^{9/4}s^{1/4}p^{5/2}(n/s)^{2/s}
    \right)^s\left(\frac{s}{fnp}\right)^{f}.\label{eq:bad-quadruple-2}
\end{align}
Since $f\le 1.1s$ we have $(s/(fnp))^f=e^{O(s)}$, and since $N_1\le s\le N_2$ we have
\[n^{9/4}s^{1/4}p^{5/2}(n/s)^{2/s}\le \lambda^{5/2}\eta^{1/4}e^{o(1)}.\]
So, if $\eta>0$ is sufficiently small then the sum in \cref{eq:bad-quadruple-2} is at most $\sum_{s=N_1}^{N_2} (1.1s+1)e^{-\Omega(s)}=o(1)$.
\end{proof}

\section{Boosting the corank}\label{sec:boosting}
In this section we prove several lemmas closely related to the strategy of Costello, Tao and Vu in \cite{CTV06}. The first lemma shows that (under certain conditions) if one extends a symmetric matrix by adding a new random row and column (this corresponds to adding a new vertex to a random graph), then the corank typically decreases. Recall the definition of the unstructured kernel property $\on{UKP}(\ell,\zeta,\eta)$ from \cref{def:UKP} (roughly speaking, this is the property that the vectors orthogonal to all but $\ell$ rows must have large level sets, allowing the possibility of a small number of exceptional rows).

\begin{lemma}\label{lem:decoupling}
Fix a constant $\eta>0$, let $g\ge1$ and let $d\ge 1$.
Consider an $n\times n$ symmetric matrix $A$ with the unstructured kernel property $\on{UKP}(2,\eta/3,\eta)$ and corank at most $g$. Consider a subset $E\subseteq\{ 1,\dots,n\}$
of size at least $n(1-\eta/3)$. Let $\mbf{x}=(x_{1},\dots,x_{n})$
be a random zero-one vector, such that the restriction $\mbf x_E$ to the entries indexed by $E$ is a uniformly random zero-one vector with exactly $d$ ones (and the restriction $\mbf x_{\ol E}$ to entries not indexed by $E$ is deterministic). Add
$\mbf{x}$ as a new row and column of $A$ (and put a zero in the
new diagonal entry) to obtain a new matrix $A'$. Then $\rank A'\ge n-g+2$ with probability at least $1-(\log d)^{O(1)}/\sqrt d$.
\end{lemma}

\begin{proof}
Let $r=\rank A$. First suppose $r\le n-1$. Without loss of generality,
we assume that the first $r$ columns $\mbf{v}_{1},\dots,\mbf{v}_{r}$
of $A$ are linearly independent. Thus, the last column $\mbf{v}_{n}$
of $A$ can be written as a linear combination of the first $r$ rows
in a unique way: $\mbf{v}_{n}=c_{1}\mbf{v}_{1}+\dots+c_{r}\mbf{v}_{r}$.
Therefore, the vector $\mbf{c}=(c_{1},\dots,c_{r},0,\dots,0,-1)$
is orthogonal to every row of $A$, hence all its level sets have size at most $(1-\eta)n$ (for convenience, we say such a vector is \emph{$\eta$-unstructured}).
Note that any restriction of $\mbf{c}$ to a set $E$ of $n(1-\eta/3)$
indices is still $(2\eta/3)$-unstructured, so by \cref{lem:linear-LO-slice}, $\mbf{c}^{T}\mbf{x}=\mbf{c}_E^{T}\mbf{x}_E+\mbf{c}_{\ol E}^{T}\mbf{x}_{\ol E}\ne0$ with probability $1-O(1/\sqrt d)$. But if $\mbf{c}^{T}\mbf{x}\ne0$
this means that adding $\mbf{x}$ as a new row (thereby appending
$x_{i}$ as a new entry to each $\mbf{v}_{i}$) increases the
rank of $A$. That is to say, $\mbf{x}$ does not lie in the
span of the rows of $A$, and by symmetry it also does not lie in
the span of the columns of $A$. So, adding $\mbf{x}$ as a
new row and column increases the rank twice, meaning that $\rank A'\ge r+2\ge n-g+2$.

On the other hand, suppose $r=n$. Observe that
\[\det A'=\mbf{x}^{T}M\mbf{x}=\mbf{x}_E^{T}M_E\mbf{x}_E+\mbf v^T\mbf x_E+x,\]
for some (non-random) $\mbf v\in \RR^E$ and $x\in \RR$, where $M$ is the adjugate matrix of $A$ and $M_E$ is its restriction to the rows and columns in $E$.

Since $AM$ is a nonzero multiple of the identity matrix, the $i$th column of $M$ is orthogonal to all rows of $A$ except the $i$th, and the difference between the $i$th and $j$th columns of $M$ is therefore orthogonal to all rows of $A$ except the $i$th and the $j$th. Recalling the definition of $\on{UKP}(2,\eta/3,\eta)$, it follows that there is $Q\subseteq\{1,\ldots,n\}$ of size $\eta n/3$ such that each column of $M$ (or row of $M$, by symmetry) not indexed by $Q$ is $\eta$-unstructured, and every difference of two rows or two columns, neither of which are indexed by $Q$, is either zero or $\eta$-unstructured.

Fix any $\eta$-unstructured column $\mbf{w}=(w_1,\dots,w_n)$ of $M$ (any column not indexed by $Q$ will do). There are at least $(\eta/3)(1-\eta)n^2$ pairs of indices $i,i'\in E\setminus Q$ such that $w_i\ne w_{i'}$. For each such pair consider the difference $\mbf v_i-\mbf v_{i'}$, where $\mbf v_i$ and $\mbf v_i$ are the $i$th and $i'$th row vectors of $M$. It follows that $\mbf v_i-\mbf v_{i'}$ is nonzero, so is $\eta$-unstructured. It follows that there are at least $(2\eta/3)(1-\eta)n^2$ pairs of indices $j,j'\in E$ such that the $j$th and $j'$th entries of $\mbf v_i-\mbf v_{i'}$ differ. For each of the resulting $\Omega(n^4)$ choices of $i,i',j,j'$, we have $m_{ij}-m_{i'j}-m_{ij'}+m_{i'j'}\neq 0$. By \cref{prop:quadratic-LO-general} it follows that with probability at least $1-(\log d)^{O(1)}/\sqrt d$ we have $\det A'\ne0$ (meaning that $\rank A'=n+1\ge n+2-g$).
\end{proof}

The second lemma in this section shows that if a random process takes nonnegative values and at each time-step tends to move towards zero, then the process is likely to end at zero. If we construct a random graph by iteratively adding many vertices with random neighbourhoods, we can combine this lemma with \cref{lem:decoupling} to show that the resulting random graph is likely to have an adjacency matrix with corank zero (i.e., be nonsingular).
\begin{lemma}\label{lem:random-walk}
Let $X_{0},\dots,X_{n}$ be a sequence of nonnegative integer random
variables satisfying the following conditions.
\begin{enumerate}
\item $X_{0}\le n/2$.
\item for all $1\le t\le n$, we have $X_{t}\le X_{t-1}+1$.
\item The sequence ``drifts towards zero'' in the sense that for any $x_{0},\dots,x_{t-1}$:
\begin{enumerate}
\item If $x_{t-1}\ne0$ we have $\Pr(X_{t}<x_{t-1}\,|\,X_{0}=x_{0},\dots,X_{t-1}=x_{t-1})\ge1-p$
\item If $x_{t-1}=0$ we have $\Pr(X_{t}=0\,|\,X_{0}=x_{0},\dots,X_{t-1}=x_{t-1})\ge1-p.$
\end{enumerate}
\end{enumerate}
Then $\Pr(X_{n}=0)\ge1-1000p$.
\end{lemma}

\begin{proof}
We may assume $p\le 1/1000$. We may also assume $n\ge 10$, because otherwise $\Pr(X_n = 0)\ge (1-p)^n\ge 1-np\ge 1-9p$.

Define $Y_{t}=(1/\sqrt p)^{X_{t}}-1$. We claim that
\begin{equation}\label{eq:walk}
    \mb{E}[Y_{t+1}|Y_t,\ldots,Y_0]\le (2\sqrt{p})Y_t + 2\sqrt{p}.
\end{equation}
To see this, we distinguish cases: if $Y_t = 0$ then $\mb{E}[Y_{t+1}|Y_t,\ldots,Y_0]\le (1/\sqrt{p}-1)p\le \sqrt{p}$, which implies \cref{eq:walk} with room to spare.
If $Y_t>0$ then we have
\[\mb{E}[Y_{t+1}+1|Y_t,\ldots,Y_0]\le \sqrt{p}(Y_t+1)+p \cdot (1/\sqrt{p})(Y_t+1),\]
which also implies \cref{eq:walk}.
We deduce from \cref{eq:walk} that
\[\mb{E}[Y_n]\le \mb{E}[Y_0](2\sqrt{p})^n + (2\sqrt{p})^{n-1}+\cdots+(2\sqrt{p})^2+ 2\sqrt{p}\le(1/\sqrt{p})^{n/2}(2\sqrt{p})^{n} + 4\sqrt{p}\le 8\sqrt{p}\]
where the final inequality follows since $p\le 1/1000$ and $n\ge 10$. We then apply Markov's inequality:
\[\Pr(X_n\neq 0) = \Pr(Y_n\neq 0) = \Pr(Y_n\ge 1/\sqrt{p}-1)\le \mb{E}[Y_n]/(1/\sqrt{p}-1)\le 9p.\qedhere\]
\end{proof}

\section{Preliminaries on degree-constrained random graphs}\label{sec:degree-constrained}

For a set $V$ and positive integers $m,k$, let $\mc K(V,m,k)$ be the uniform distribution on graphs with vertex set $V$, exactly $m$ edges, and minimum degree at least $k$. The reason we are interested in this random graph model is the following symmetry property of the $k$-core. This property is well-known and appears implicitly for instance in \cite{CW06}.
\begin{lemma}\label{lem:rotate-core}
For any $p$, let $G\sim \GG(n,p)$, let $V$ be the vertex set of $\core_k(G)$ and let $m$ be the number of edges in $\core_k(G)$. If we condition on any outcome of $V,m$, then the conditional distribution of $\core_k(G)$ is $\mc K(V,m,k)$.
\end{lemma}
\begin{proof}
Consider any two graphs $H_1,H_2$ on the vertex set $V$ with $m$ edges and minimum degree at least $k$. For any outcome of $G$ yielding $\core_k(G)=G[V]=H_1$, we can simply replace $G[V]$ with $H_2$ to obtain an outcome of $G$ yielding $\core_k(G)=H_2$ (iteratively deleting vertices with degree less than $k$ yields $G[V]$ in both cases). This means that $H_1$ and $H_2$ are equally likely to occur as $\core_k(G)$.
\end{proof}

\begin{remark}\label{rmk:hitting-time}
The conclusion of \cref{lem:rotate-core} holds not just for $G\sim \GG(n,p)$ but for any ``sufficiently symmetric'' random graph. For example, it holds when $G$ is a random graph with a specified number of edges. Actually, it is straightforward to adapt the entire proof of \cref{thm:main} to this setting. Indeed, the only additional inputs we need are some very weak estimates on the number of vertices and edges in $\core_k(G)$, and analogues of the expansion estimates in \cref{lem:expansion-estimates-1,lem:expansion-estimates-2}.
\end{remark}

For convenience, we write $\mc K(n,m,k)$ instead of $\mc K(\{1,\dots,n\},m,k)$. The degree sequence of a typical $G\sim \mc K(n,m,k)$ was studied in work of Cain and Wormald~\cite{CW06}, as follows\footnote{Strictly speaking part (2) of \cref{lem:core-degree-sequence} does not follow from the statements of \cite[Theorem~2 and ~Lemma~1]{CW06}, but it is a much simpler fact than (1) and can be easily derived using Cain and Wormald's methods.}.

\begin{definition}\label{def:CW}
Consider integers $k,m,n$ satisfying $k\ge 3$ and $m\ge kn/2$. Let $Z_k(\lambda) = \sum_{t=k}^\infty\lambda^t/t!$, and let $\lambda$ be the unique solution to $\lambda Z_k'(\lambda)/Z_k(\lambda)=2m/n$. (This unique solution $\lambda$ is easily seen to exist; see e.g.~\cite[Eq.~(4)]{CW06}). Define $\rho_t =(\lambda^t/t!)/Z_k(\lambda)$ for $t\ge k$; these values can be interpreted as probability masses associated with a truncated Poisson distribution.
\end{definition}

\begin{lemma}[{\cite[Theorem~2 and ~Lemma~1]{CW06}}]\label{lem:core-degree-sequence}
Fix constants $k\ge 3$ and $\varepsilon > 0$. Suppose $\varepsilon n \le m-kn/2\le n/\varepsilon$ and let $G\sim \mc{K}(n,m,k)$. Then, using the notation of \cref{def:CW}, whp the following properties hold.
\begin{enumerate}
\item for all $t\ge 3$, there are $\rho_t n+O(n^{3/4})$ vertices of degree $t$.
\item all vertices have degree at most $\log n$.
\end{enumerate}
\end{lemma}

Recall from the outline in \cref{sec:outline} that we will need to extract some high-degree vertices from the $k$-core, which we will later use to ``boost'' the corank. So, we will also need to understand the degree sequence of the subgraph of a typical outcome of $G\sim \mc K(n,m,k)$ obtained by removing a subset of high-degree vertices. This degree sequence is a little more complicated to describe.
\begin{definition}\label{def:CW-mod}
Consider integers $\Delta,k,m,n$ satisfying $\Delta,k\ge 3$ and $m\ge kn/2$. Let $\lambda,Z_k(\lambda),\rho_t$ be as in \cref{def:CW}, and further define:
\begin{align*}
\beta = \alpha\sum_{t=\Delta}^\infty\rho_t,\qquad
\gamma = \frac{\alpha n}{2m}\sum_{t=\Delta}^\infty t\rho_t = \frac{\alpha}{Z_k'(\lambda)}\sum_{t=\Delta-1}^{\infty}\frac{\lambda^t}{t!},
\end{align*}
\begin{align*}
\zeta_{j,t}=\binom jt\gamma^{j-t}(1-\gamma)^t,\qquad\delta_t  = (1-\alpha)\sum_{j=k}^\infty\rho_j \zeta_{j,t},\qquad \delta_t'  = \alpha\sum_{j=k}^{\Delta-1}\rho_j\zeta_{j,t}
\end{align*}
for $j,t\ge 0$.
\end{definition}
\begin{lemma}\label{lem:bulk-sequence}
Fix $\Delta\ge k\ge 3$ and $\varepsilon,\alpha > 0$. Suppose $\varepsilon n \le m-kn/2\le n/\varepsilon$ and let $G\sim \mc{K}(n,m,k)$. Let $V=\{1,\dots,n\}$, $S=\{1,\dots,\floor{\alpha n}\}$ and $T=\{v\in S:\deg_G(v)\ge \Delta\}$. Then whp the following hold.
\begin{enumerate}
    \item $|T|=\beta n+O(n^{3/4})$.
    \item For all $t\ge 0$, there  are $\delta_t n+O(n^{4/5})$ vertices $v\in V\setminus S$ with $\deg_{V\setminus T}(v)=t$.
    \item For all $t\ge 0$, there  are $\delta_t' n+O(n^{4/5})$ vertices $v\in S\setminus T$ with $\deg_{V\setminus T}(v)=t$.
\end{enumerate}
\end{lemma}
For the proof of \cref{lem:bulk-sequence} (and for other computations in this and later sections), we work with a random graph model called the \emph{configuration model}.

\begin{definition}
For a degree sequence $\mbf d=(d_1,\dots,d_n)$, consider a set of $r=d_1+\dots+d_n$ points, grouped into $n$ labelled ``buckets'' of sizes $d_1,\dots,d_n$. A \emph{configuration} is a perfect matching on the $r$ points, consisting of $r/2$ disjoint edges. Now, given a configuration, contracting each of the buckets to a single point gives rise to a multigraph with degree sequence $d_1,\dots,d_n$ (where we use the convention that loops contribute 2 to the degree of a vertex).
\end{definition}
If we consider the multigraph $G^*$ arising from a uniformly random configuration, and condition on the event that this multigraph is a simple graph, then we obtain the uniform distribution on graphs with degree sequence $\mbf d$. Moreover, if $d_1^2+\dots+d_n^2=O(n)$ (as is typically the case for the degree sequence of any subgraph of $G\sim \mc K(n,m,k)$, as follows from \cref{lem:core-degree-sequence}), it is well known that this simplicity probability is $\Omega(1)$. So, any property that holds whp for $G^*$ also holds whp for a uniformly random graph with degree sequence $\mbf d$.

\begin{proof}[Proof of \cref{lem:bulk-sequence}]
Let $\mbf d=(d_1,\dots,d_n)$ be the degree sequence of $G\sim\mc{K}(n,m,k)$. By symmetry, all reorderings of $\mbf d$ are equally likely to occur, so (1) holds whp by \cref{lem:core-degree-sequence} and a Chernoff bound for hypergeometric distributions (see for example \cite[Theorem~2.10]{JLR00}). Similarly, $\mbf d$ has the following properties whp.
\begin{itemize}
    \item For all $t\ge k$, there are $(1-\alpha)\rho_t n+O(n^{3/4})$ vertices $v\in V\setminus S$ with $d_v=t$.
    \item For all $k\le t<\Delta$, there are $\alpha\rho_t n+O(n^{3/4})$ vertices $v\in S\setminus T$ with $d_v=t$.
    \item For all $t\ge \Delta$, there are $\alpha\rho_t n+O(n^{3/4})$ vertices $v\in T$ with $d_v=t$.
\end{itemize}

Condition on an outcome of $\mbf d$ satisfying the above properties and such that each $d_i\le \log n$ (this degree sequence determines $T$). Now, $G$ is now a uniformly random graph with degree sequence $\mbf d$, and $d_1+\dots+d_n=2m$. Consider a set of $2m$ points grouped into buckets of sizes $d_1,\dots,d_n$. Let $\sigma$ be a uniformly random bijection from $\{1,\dots,2m\}$ into these $2m$ points, so the perfect matching with edges $\sigma(1)\sigma(2),\; \sigma(3)\sigma(4),\;\dots,\;\sigma(2m-1)\sigma(2m)$ is a uniformly random configuration. Let $G^*$ be the multigraph obtained by contracting this configuration. It suffices to show that $G^*[V\setminus T]$ satisfies properties (2) and (3) whp.

For each $t\in \NN$, let $X_t$ (respectively $X_t'$) be the number of vertices $v\in V\setminus S$ (respectively, $v\in S\setminus T$) with $\deg_{V\setminus T}(v)=t$ in $G^*$. Modifying $\sigma$ by a transposition affects $X_t$ (respectively, $X_t'$) by at most 4, so by a concentration inequality for random permutations (see for example \cite[Eq.~(29)]{McD98}), whp $X_t=\E X_t+O(n^{3/4})$ and $X_t'=\E X_t'+O(n^{3/4})$ for each $t$. Therefore, it suffices to show that $\E X_t=\delta_t n+O(n^{3/4}(\log n)^3)$ and $\E X_t'=\delta_t' n+O(n^{3/4}(\log n)^3)$ for each $t$.

Consider the set $P_T$ of all the points in buckets corresponding to vertices in $T$. The number of such points is
\[\sum_{t\ge\Delta}\frac{t}{Z_k(\lambda)}\frac{\lambda^t}{t!}\alpha n + O(n^{3/4}(\log n)^2).\]
That is to say, the fraction of the $2m$ points occupied by $P_T$ is equal to $\gamma+O(n^{-1/4}(\log n)^2)$. Now, for a vertex $v\in V\setminus T$, there are $d_v\le \log n$ points in the bucket corresponding to $v$ (let $b_v$ be the set of points in this bucket), and the number of points in this bucket that are matched with points in $P_T$ is very nearly binomially distributed with parameters $d_v$ and $1-\gamma$. Indeed, the only deviation from this distribution comes from the fact that if one knows whether some subset of points in $b_v$ are matched with points in $P_T$, this very slightly biases (by a factor of $1+O(1/n)$) the probability that a further point in $b_v$ is matched with a point in $P_T$. The probability that $v$ has exactly $d_v-t$ neighbours in $T$ is therefore $\zeta_{d_v,t}+O((\log n)^2/n^{-1/4}+\log n/n)$. Linearity of expectation then yields the desired estimates.
\end{proof}

Next, we observe that in the setting of \cref{lem:bulk-sequence}, the subgraph $G[V\setminus T]$ is uniformly random given its degree sequence.
\begin{lemma}\label{lem:rotate-bulk}
For any integers $m,k,\Delta$ and sets $V,S$, let $G\sim \mc K(V,m,k)$ and let $T$ be the set of vertices in $S$ which have at least $\Delta$ neighbours in $V$. If we condition on any outcome of $V\setminus T$ and any outcome of the degree sequence $\mbf d=(\deg_{V\setminus T}(v))_{v\in V\setminus T}$, then the conditional distribution of $G[V\setminus T]$ is uniform over graphs on the vertex set $V\setminus T$ with degree sequence $\mbf d$.
\end{lemma}
\begin{proof}
Consider any two graphs $H_1,H_2$ on the same vertex set with the same degree sequence. For any outcome of $G$ yielding $G[V\setminus T]=H_1$, we can simply replace $G[V\setminus T]$ with $H_2$ to obtain an outcome of $G$ yielding $G[V\setminus T]=H_2$. So, $H_1$ and $H_2$ are equally likely to occur as $G[V\setminus T]$.
\end{proof}

The final fact we record in this section is that sparse random graphs with given degree sequences are typically well-approximated by Galton--Watson trees.
\begin{definition}\label{def:BS-1}
For a probability distribution $\mu$ with nonnegative integer support, let $\tilde{\mu}$ be the probability distribution defined by \[\tilde{\mu}(t)=(t+1)\mu(t+1)\big/\sum_{\ell=0}^\infty \ell \mu(\ell)\]
for positive integer $t$. A Galton--Watson tree $T_\mu$ with \emph{degree distribution} $\mu$ is a random rooted tree obtained by a Galton--Watson branching process, where the root has offspring size distribution $\mu$, and all other generations have offspring size distributions $\tilde{\mu}$. Let $T_\mu^{(r)}$ consist of the first $r$ generations of $T_\mu$ (with all vertices unlabelled except the root).

Also, for a graph $G$, let $G^{(r)}$ be the random unlabelled rooted graph obtained by fixing a uniformly random vertex $v$ as the root, and including the subgraph of all vertices within distance $r-1$ of $v$. Let $\on d^{(r)}_\mu(G)$ be the total variation distance $\on d_{\mr{TV}}(G^{(r)},T_\mu^{(r)})$ between $G^{(r)}$ and $T_\mu^{(r)}$ (if $G$ is a random graph, then $\on d^{(r)}_\mu(G)$ is a random variable).
\end{definition}

\begin{lemma}\label{lem:BS}
Let $\mu$ be a probability distribution with nonnegative integer support and fix any positive integer $r$. Let $\mbf d=(d_1,\dots,d_n)$ be a degree sequence such that each $d_v\le \log n$, and such that $|\{v:d_v=t\}|=\mu(t)n+O(n^{4/5})$ for each $t\ge 0$. Let $G$ be a uniformly random graph with degree sequence $\mbf d$. Then, whp $\on d^{(r)}_\mu(G)=o(1)$.
\end{lemma}
\begin{proof}
Consider a uniformly random configuration with degree sequence $\mbf d$ (obtained via a random permutation $\sigma$, as in \cref{lem:bulk-sequence}). Note that for $t\ge 1$, the number of points in buckets of size $t$ is $\tilde{\mu}(t-1) n+O(n^{4/5}\log n)$.

Let $G^*$ be the multigraph obtained by contracting our random configuration. For a vertex $v$, let $G^*(v)$ be the rooted graph consisting of all vertices within distance $r-1$ of $v$. For each possible outcome $H$ of $G^*(v)$ which is a tree (there are at most $(\log n)^r$ such outcomes), let $X_H$ be the number of vertices $v$ for which $G^*(v)=T$. Also, let $X_0$ be the number of vertices $v$ for which $G^*(v)$ is not a tree. Modifying $\sigma$ by a transposition changes $X_0$ and each $X_H$ by at most $4(\log n)^k=n^{o(1)}$, so by a concentration inequality (see for example \cite[Eq.~(29)]{McD98}), whp $X_0=\E X_H+O(n^{5/6})$ and $X_H=\E X_H+O(n^{5/6})$ for each $H$. Now, to finish the proof we claim that $\E X_0=O(n^{4/5}\log n+(\log n)^r/n)$ and $\E X_H=\Pr(T^{(r)}_\mu=H)+O(n^{4/5}\log n+(\log n)^r/n)$ for all $H$. This follows from a routine computation in the configuration model (choose a random vertex, and iteratively explore the neighbourhood of that vertex to a depth of $r$, at each step exposing a new matched pair in our random configuration).
\end{proof}

\section{Initial rank estimate}\label{sec:rank}

In this section we apply the machinery of Bordenave, Lelarge and Salez~\cite{BLS11}.

\begin{definition}\label{def:BS-2}
Say that a sequence of random graphs $G_1,G_2,\dots$ is \emph{$\mu$-convergent} if, in the notation of \cref{def:BS-1}, the sequence of random variables $\on d^{(r)}_\mu(G_1)$,$\on d^{(r)}_\mu(G_2),\dots$ converges to zero in probability, for every fixed $r$. For a function $f:\mb Z_{\ge 0}\to \RR_{\ge 0}$ with $\sum_{i=0}^\infty f(i)<\infty$,  let $\varphi_f=[0,1]\to \RR$ be the function $\sum_{i=0}^\infty f(i)x^i$, and let $M_f:[0,1]\to \RR$ be the function defined by
\[M_f(x) = x\varphi_f'(1-x) + \varphi_f(1-x) + \varphi_f(1-\varphi_f'(1-x)/\varphi_f'(1)) - 1.\]
\end{definition}

\begin{theorem}[{See \cite[Theorem~13 and Eq.~(19)]{BLS11}}]\label{thm:BLS}
Let $\mu$ be a probability distribution with nonnegative integer support and finite second moment, and let $G_1,G_2,\dots$ be a $\mu$-convergent sequence of graphs with adjacency matrices $A_1,A_2,\dots$. Then
\[\limsup_{n\to\infty}\frac{1}{n}\mb{E}[\corank A_n]\le\max_{x\in[0,1]}M_{\mu}(x).\]
\end{theorem}

Using \cref{thm:BLS} together with \cref{lem:core-degree-sequence,lem:BS}, one can prove that if $k\ge 3$ and $|V|,m=\Theta(n)$ then a random graph $G\sim \mc K(V,m,k)$ has corank $o(n)$ whp. We will need a variation on this fact, in which some high-degree vertices are first removed.

\begin{lemma}\label{lem:rank}
Fix  $k\ge 3$ and $\varepsilon,\alpha,\Delta > 0$, such that $\alpha$ is sufficiently small with respect to $\varepsilon$ and $\Delta$ is sufficiently large with respect to $\alpha$. Consider sets $S\subseteq V$ and an integer $m$ such that $|V|=\Theta(n)$ and $\varepsilon n\le m-k|V|/2\le n/\varepsilon$ and $|S|=\floor{\vphantom{x^{x^x}}\alpha |V|}$, and consider a random graph $G\sim \mc K(V,m,k)$. Let $T=\{v\in S:\deg_G(v)\ge \Delta\}$ be the set of vertices in $S$ with degree at least $\Delta$. Then, whp the corank of the adjacency matrix of $G[V\setminus T]$ is at most $|T|/8$.
\end{lemma}

\begin{remark}\label{rem:np-infty}
Essentially all parts of our proof of \cref{thm:main} straighforwardly generalise to the setting where $np$ slowly tends to infinity (we only need to consider the regime where say $np\le 2\log n$; if $np$ grows faster than this, then it is well-known that whp $\core_k(G)=G$ is nonsingular). The exception is \cref{lem:rank}, since it is proved using \cref{thm:BLS} (which can only ever apply to sequences of graphs for which the number of edges is proportional to the number of vertices). It may be possible to quantify the methods of Bordenave, Lelarge and Salez, to obtain some version of \cref{thm:BLS} that applies when $n=o(m)$, but this is actually not necessary, due to various simplifications that can be made in the regime $np\to\infty$.

We (very) briefly sketch how to adapt and simplify our proof of \cref{thm:main}: if $np \to \infty$ then typically there are only $o(n)$ vertices outside of $\core_k(G)$, and (by estimating the expected size of the kernel of the adjacency matrix of $G$, interpreted as a matrix over $\mb F_2$; see \cite{Fer20}), it is not difficult to show that $\corank G=o(n)$ whp. So, we can simply fix a set $T$ of $\beta n$ vertices, for some small $\beta$ (without worrying about the set $S$). Since $G[V\setminus T]$ is an Erd\H os--R\'enyi random graph on its vertex set, whp $\core_k(G)-T$ has corank $o(n)$, and it is also easy to see that whp all but $o(n)$ vertices in $T$ have $(1-\beta+o(1))np\to\infty$ neighbours in $\core_k(G)-T$. Then we finish with a rank-boosting argument similar to one in \cref{sec:finishing}.

There are various other minor changes that would have to be made in various parts of this paper to handle the case $np\to \infty$ (for example, we would need a slightly more involved analysis for the unstructured kernel property than is currently in \cref{lem:UKP}). However, since the analogue of \cref{thm:main} in the regime $np\to\infty$ was recently proved by Glasgow~\cite{Gla21} (for $k=3$, but the extension to all $k\ge 3$ is immediate), we do not elaborate further on these changes.
\end{remark}

The proof of \cref{lem:rank} basically amounts to analysis of the function $M_\mu$, for an appropriate distribution $\mu$. Indeed, we will be able to prove \cref{lem:rank} via a very short deduction from the following lemma.

\begin{lemma}\label{lem:rank-analysis}
Fix $\Delta\ge k\ge 3$ and $0<\varepsilon,\alpha <1/2$, such that $\Delta$ is sufficiently large with respect to $k,\varepsilon,\alpha$. Suppose $\varepsilon n \le m-kn/2\le n/\varepsilon$. For $t\in \mb Z_{\ge 0}$, let $\mu(t)=(\delta_t+\delta_t')/(1-\beta)$, in the notation of \cref{def:CW-mod}. Then $\sup_{x\in[0,1]}M_\mu(x)\le\beta/16$.
\end{lemma}

\begin{proof}[Proof of \cref{lem:rank} given \cref{lem:rank-analysis}]
First note that it suffices to consider the case where $V=\{1,\dots,n\}$. Also, note that $\mu$ defines a probability distribution (one can directly compute that $\sum_{t=0}^\infty\mu(t)=1$, or deduce this indirectly from \cref{lem:bulk-sequence}). Let $A$ be the adjacency matrix of $G[V\setminus T]$. By part (1) of \cref{lem:bulk-sequence}, we have $|T|=\beta n+O(n^{3/4})$ whp. By part (2) of \cref{lem:core-degree-sequence} and parts (2) and (3) of \cref{lem:bulk-sequence}, together with \cref{lem:rotate-bulk,lem:BS}, we see that $G[V,T]$ is $\mu$-convergent, so $\E \corank A\le (\beta/16)n+o(n)$ by \cref{thm:BLS}.

Now, we finish the proof by noting that $\corank A$ is tightly concentrated: indeed, if we condition on the degree sequence $\mbf d$ of $G[V\setminus T]$, then \cref{lem:rotate-bulk} tells us that $G[V\setminus T]$ is uniform over graphs with degree sequence $\mbf d$. Consider a random configuration with degree sequence $\mbf d$, defined in terms of a random permutation $\sigma$ (as in the proof of \cref{lem:bulk-sequence}), and let $A^*$ be the adjacency matrix of the corresponding contracted multigraph. Then, modifying $\sigma$ by a transposition changes at most 4 rows of $A^*$, so changes $\corank A^*$ by at most 4. Therefore, by a concentration inequality (see for example \cite[Eq.~(29)]{McD98}), we have $\corank A^*=\E \corank A^*+o(n)$ whp and therefore $\corank A=\E \corank A^*+o(n)$ whp, since $G^*$ is simple with probability $\Omega(1)$. Since $\corank A^*$ always lies between $0$ and $n$, it follows that $\E \corank A=\E \corank A^*+o(n)$, and therefore $\corank A\le |T|/8$ whp.
\end{proof}

We will spend the rest of this section proving \cref{lem:rank-analysis}.
\begin{proof}[Proof of \cref{lem:rank-analysis}]
First, recall the definitions of $\lambda,Z_k(\lambda),\gamma$ from \cref{def:CW,def:CW-mod}, and define $\nu\colon\mb{Z}_{\ge 0}\to\mb{R}_{\ge 0}$ by
\[\nu(t) = \frac{(1-\gamma)^t\lambda^t}{(1-\beta)Z_k(\lambda)t!}((\mbm{1}_{t\ge k}-\alpha\mbm{1}_{t\ge\Delta}) + \gamma\lambda(\mbm{1}_{t+1\ge k}-\alpha\mbm{1}_{t+1\ge\Delta})).\]
This does not define a probability distribution, but we do have $\sum_{i=0}^\infty i^t\nu(i)<\infty$ for all $t\in \NN$. The idea is that $\nu$ is an approximation of $\mu$ that is more tractable to analyse.

For any $f:\mb Z_{\ge 0}\to \RR_{\ge 0}$ satisfying $\sum_{i=0}^\infty i^2 f(i)<\infty$ and $f(0) = f(1) = 0$, we can differentiate term-by-term to compute
\[M_f'(0) = \frac{\varphi_f''(1)}{\varphi_f'(1)}\varphi_f'(0) = 0\]
By \cite[Theorem~13]{BLS11} and the remarks following, if $\varphi_f''(x)$ is log-concave on $[0,1]$ then the global maximum of $M_f$ over $[0,1]$ is the first local extremum, namely $x = 0$.

The remainder of the proof proceeds in two steps. In one step, we show that $\varphi_{\nu}''(x)$ is log-concave, hence $M_{\nu}(x)\le M_{\nu}(0)$. In the other step, we show that $\nu$ and $\mu$ are ``close'' in the sense that
\begin{equation}\label{eq:M-transfer}
\sup_{x\in[0,1]}|M_\mu(x)-M_{\nu}(x)|\le\beta/32.
\end{equation}
Thus, for all $x\in[0,1]$, it will follow that
\[M_\mu(x)\le M_{\nu}(x)+\beta/32\le M_{\nu}(0)+\beta/32\le M_\mu(0)+\beta/16 = \beta/16\]
as desired. Here we used the fact that $M_\mu(0) = \varphi_\mu(1)-1 = 0$.
\medskip

\noindent \textbf{Step 1: Small difference.}
We first show \cref{eq:M-transfer}. Recall from \cref{def:CW-mod} that $\zeta_{j,t}$ is the probability that a $\on{Binomial}(j,1-\gamma)$ random variable is equal to $t$ (so $\zeta_{j,t}=0$ if $j<t$), write $x\lor y=\max(x,y)$, and observe that
\begin{align*}
(1-\beta)Z_k(\lambda)(\mu(t)-\nu(t))&=\left(\sum_{j=k\lor t}^\infty\frac{\lambda^j}{j!}\zeta_{j,t}-\alpha\sum_{j=\Delta\lor t}^\infty\frac{\lambda^j}{j!}\zeta_{j,t}\right)\\
&\qquad\qquad-\left((\mbm{1}_{t\ge k}-\alpha\mbm{1}_{t\ge\Delta})\zeta_{t,t}\frac{\lambda^t}{t!} + (\mbm{1}_{t+1\ge k}-\alpha\mbm{1}_{t+1\ge\Delta})\zeta_{t+1,t}\frac{\lambda^{t+1}}{(t+1)!}\right)\\
&= \sum_{j=k\lor(t+2)}^\infty\frac{\lambda^j}{j!}\zeta_{j,t}-\alpha\sum_{j=\Delta\lor (t+2)}^\infty\frac{\lambda^j}{j!}\zeta_{j,t}.
\end{align*}
Hence, using the assumption $\Delta\ge k$, we have
\begin{align*}
|(1-\beta)Z_k(\lambda)(\mu(t)-\nu(t))|&\le \sum_{j=k\lor(t+2)}^\infty\frac{\lambda^j}{j!}\zeta_{j,t}= \sum_{j=k\lor (t+2)}^\infty\frac{\lambda^j}{t!(j-t)!}(1-\gamma)^t\gamma^{j-t}\le \frac{2e^\lambda\lambda^t\gamma^2}{t!}.
\end{align*}
For the rest of the proof we use asymptotic notation, letting $\Delta\to \infty$ while $\varepsilon,k,\alpha$ are fixed (so $\beta,\gamma=o(1)$ and $\lambda=O(1)$). In asymptotic notation, we have just shown that $|\mu(t)-\nu(t)|=O(\gamma^2\lambda^t/t!)$, and therefore
$|\varphi_\mu(x)-\varphi_{\nu}(x)|,|\varphi_\mu'(x)-\varphi_{\nu}'(x)|\le O(\gamma^2)$ for all $x\in[0,1]$.
Also, note that $\mu(t)\le e^\lambda \lambda^t/t!$, so $\varphi_\mu'(x)=O(1)$ for all $x\in [0,1]$, and $\varphi_\mu'(1)\ge \mu(k)=\Omega(1)$. So, we have
\[\varphi_\mu\left(1-\frac{\varphi_\mu'(1-x)}{\varphi_\mu'(1)}\right)-\varphi_\nu\left(1-\frac{\varphi_\nu'(1-x)}{\varphi_\nu'(1)}\right)=O(\gamma^2)+O\bigg(\frac{\varphi_\mu'(1-x)}{\varphi_\mu'(1)}-\frac{\varphi_\nu'(1-x)}{\varphi_\nu'(1)}\bigg)=O(\gamma^2)\]
for each $x\in[0,1]$.
Recalling that $M_f(x) = x\varphi_f'(1-x)+\varphi_f(1-x)+\varphi_f(1-\varphi_f'(1-x)/\varphi_f'(1)) - 1$, it follows that $\sup_{x\in[0,1]}|M_\mu(x)-M_{\nu}(x)|=O(\gamma^2)$. Finally, we claim that $\gamma^2=o(\beta)$, which would imply that if $\Delta$ is sufficiently large in terms of the other parameters, then \cref{eq:M-transfer} holds. To see this, we observe that $\gamma=O(\E[X\mbm1_{X\ge \Delta}])$ and $\beta=\Omega(\E[\mbm1_{X\ge \Delta}])$, for $X\sim\on{Poisson}(\lambda)$, and $\E[X\mbm1_{X\ge \Delta}]\le (\E X^4)^{1/4}(\E[\mbm1_{X\ge \Delta}])^{3/4}=O((\E[\mbm1_{X\ge \Delta}])^{3/4})$ by H\"older's inequality. 
\medskip

\noindent \textbf{Step 2: Log-concavity.}
Next we show the desired log-concavity of $\varphi_{\nu}''(x)$, for $x\in [0,1]$. By rescaling $\nu$ and $x$, we see that it suffices to show $\phi''$ is log-concave in the interval $x\in [0,(1-\gamma)\lambda]$, where $\phi:[0,1]\to \RR$ is defined by
\[\phi(x) = \sum_{t=k-1}^\infty\frac{x^t}{t!}((\mbm{1}_{t\ge k}-\alpha\mbm{1}_{t\ge\Delta}) + \gamma\lambda(\mbm{1}_{t+1\ge k}-\alpha\mbm{1}_{t+1\ge\Delta})).\]

Now, it is well-known (see  for example \cite[Lemma~3]{BB05}) that if a nonnegative function $h$ is log-concave on an interval $[a,b]$, then its antiderivative $x\mapsto \int_a^xh(y)\on dy$ is also log-concave on that same interval. Since $k\ge 3$, it therefore suffices to show that the $(k-1)$-fold derivative $f = \phi^{(k-1)}$ is log-concave on the interval $[0,(1-\gamma)\lambda]$. We claim that in fact it suffices to prove that $g=f+1$ is log-concave on this interval. Indeed, first observe that $f$ is convex (e.g., by differentiating term-by-term). If we could show that $g = f+1$ were log-concave on $[0,(1-\gamma)\lambda]$, it would follow that for any $x,y\in [0,(1-\gamma)\lambda]$,
\begin{align*}
f(x)f(y) &= (1+f(x))(1+f(y))-f(x)-f(y)-1\\
&\le\bigg(1+f\bigg(\frac{x+y}{2}\bigg)\bigg)^2-2f\bigg(\frac{x+y}{2}\bigg)-1= f\bigg(\frac{x+y}{2}\bigg)^2
\end{align*}
implying that $f$ itself is log-concave on $[0,(1-\gamma)\lambda]$, as desired. So, we study the function $g$, which we may write as
\[g(x) = (1+\gamma\lambda)\sum_{t=0}^{\Delta-k-1}\frac{x^t}{t!} + (1+\gamma\lambda(1-\alpha))\frac{x^{\Delta-k}}{(\Delta-k)!}+(1+\gamma\lambda)(1-\alpha)\sum_{t = \Delta-k+1}^\infty\frac{x^t}{t!}.\]
Let $h(x) = e^x-g(x)/(1+\gamma\lambda)$, so
\[h(x) = \frac{\alpha\gamma\lambda}{1+\gamma\lambda}\cdot\frac{x^{\Delta-k}}{(\Delta-k)!} + \alpha\sum_{t=\Delta-k+1}^\infty\frac{x^t}{t!}.\]
We have
\begin{align*}
\frac{\on{d}^2}{\on{d}x^2}\log\left(\frac{g(x)}{1+\gamma\lambda}\right) &= \frac{(e^x-h(x))(e^x-h''(x))-(e^x-h'(x))^2}{(e^x-h(x))^2}\\
&= -\frac{e^x(h(x)+h''(x)-2h'(x)) + (h'(x)^2-h(x)h''(x))}{(e^x-h(x))^2}.
\end{align*}
Furthermore, for sufficiently large $\Delta$ we have $h(x)-2h'(x)+h''(x)\ge 0$ for $x\in [0,(1-\gamma)\lambda]$. To see this, observe that $h(x)$ has a nonzero coefficient of $x^t$ only when $t\ge \Delta-k$, and note that $x^t-2tx^{t-1}+t(t-1)x^{t-2} = x^{t-2}(x^2-tx+t(t-1))\ge 0$ when $x\le t-\sqrt{t}$. Thus, it suffices to show that $h(x)h''(x)\le h'(x)^2$; that is, it suffices to show that $h$ is log-concave on $[0,(1-\gamma)\lambda]$. We will show that in fact $h$ is log-concave everywhere.

To this end, it suffices to show that the $(\Delta-k)$-fold derivative $r=h^{(\Delta-k)}$ is log-concave. We write
\[r(x) = \frac{\alpha\gamma\lambda}{1+\gamma\lambda} + \alpha\sum_{t=1}^\infty\frac{x^t}{t!}\]
and now we can use the same trick as before: $r$ is convex, so we finish the proof by observing that the function $r+\alpha/(1+\gamma\lambda)$, which is precisely the exponential function $x\mapsto e^x$, is log-concave.
\end{proof}

% \begin{remark}
% Although we wrote the statement of \cref{lem:rank-analysis} with the assumption $k\ge 3$, it actually still holds if $k=2$. Indeed, instead of showing that $g=\phi^{(k-1)}+1$ is log-concave, we can directly show that $\phi''$ is log-concave: we compute
% \[\phi''(x)=(1+\gamma\lambda)\sum_{t=0}^{\Delta-4}\frac{x^t}{t!} + (1+\gamma\lambda(1-\alpha))\frac{x^{\Delta-3}}{(\Delta-3)!}+(1+\gamma\lambda)(1-\alpha)\sum_{t = \Delta-2}^\infty\frac{x^t}{t!}.\]
% which has essentially the same form as $g$, and the rest of the proof can proceed in the same way.
% \end{remark}

\section{Extracting high-degree vertices with random neighbourhoods}

In the last section we showed that if we take a random graph $G\sim \mc K(V,m,k)$ and remove a set of high-degree vertices $T$, then the result $G[V\setminus T]$ typically has small corank. As sketched in \cref{sec:outline}, the next step is to add back the high-degree vertices, and show that they ``boost'' the corank to zero. Unfortunately, after deleting the high-degree vertices we are not yet in a position to apply the rank-boosting lemmas in \cref{sec:boosting}, mainly because $G[V\setminus T]$ does not satisfy the unstructured kernel property defined in \cref{sec:level-sets} (in the language of \cref{lem:UKP}, $G[V\setminus T]$ is not ``good''). So, we need to strategically add back a few of the vertices in $T$ (but we need to be careful to do so in such a way that we retain useful randomness). Say that a pair of vertices $(u,v)$ in a (multi)graph is \emph{$r$-joined} if $u\ne v$ and there is a path of length at most $r$ between $u$ and $v$, or if $u=v$ and there is a cycle of length at most $r$ containing $u=v$.

\begin{lemma}\label{lem:extraction}
Fix  $\Delta\ge k\ge 3$ and $\varepsilon,\alpha > 0$, such that $\Delta$ is sufficiently large with respect to $\varepsilon,\alpha$. Consider sets $S\subseteq V$ and an integer $m$ such that $|V|=\Theta(n)$ and $\varepsilon n\le m-k|V|/2\le n/\varepsilon$ and $|S|=\floor{\vphantom{x^{x^x}}\alpha |V|}$, and consider a random graph $G\sim \mc K(V,m,k)$.
\begin{itemize}
\item Let $T=\{v\in S:\deg_G(v)\ge \Delta\}$.
\item Let $B_{\mr{bias}}\subseteq V\setminus T$ be the set of vertices which either have $\deg_{V\setminus T}(v)<k$ or are within distance 2 of such a vertex (in $G[V\setminus T]$).
\item Let $T_{\mr{bad}}$ be the set of vertices appearing in a pair $(u,v)\in T^2$ which is $6$-joined in $G[T\cup B_{\mr{bias}}]$.
\item Let $E = V\setminus(B_{\mr{bias}}\cup S)$.
\item Let $T_{\mr{low}} = \{v\in T: \deg_E(v)<\sqrt{\Delta}\}$.
\item Let $T' = T\setminus (T_{\mr{bad}}\cup T_{\mr{low}})$.
\end{itemize}
Then, the following hold.
\begin{enumerate}
\item For every superset $U\supseteq V\setminus T'$, the graph $G[U]$ has minimum degree at least $2$ and has no pair of degree-2 vertices within distance 4 of each other.
\item Every vertex $v\in T'$ has $\deg_{E}(v)\ge \sqrt \Delta$.
\item Whp $|T'|\ge 7|T|/8$.
\item Suppose we reveal the sets $T',E$, and reveal the status of every edge in $G$ not between $T'$ and $E$, and reveal the degree sequence $(\deg_{E}(v))_{v\in T'}$. Then, conditionally, for each $v\in T'$, the neighbourhood $N_{E}(v)$ is a uniformly random subset of $\deg_{E}(v)$ vertices of $E$, and all such neighbourhoods are jointly independent.
\end{enumerate}
\end{lemma}

Some of the parts of \cref{lem:extraction} are straightforward. In particular, (2) is immediate by the definitions of $T_{\mr{low}}$ and $T'$. Part (1) is essentially true by definition as well.

\begin{proof}[Proof of \cref{lem:extraction}(1)]
Consider $U\supseteq V\setminus T'$. First, note that $T_{\mr{bad}}\subseteq U$, so the vertices in $U\cap T$ have all of their (at least $\Delta\ge k\ge 3$) neighbours inside $U$. So, the only vertices that could possibly have degree less than 3 are the vertices in $V\setminus T$.

Now, since $k\ge 3$, if a vertex $v\in V\setminus T$ has degree at most 2 in $G[U]$, then $v\in B_{\mr{bias}}$. Such a $v$ has exactly one neighbour $x_v$ in $T\setminus U$ (if it had more than one neighbour in $T$, all of those neighbours would be in $T_{\mr{bad}}\subseteq U$), so in fact this is only possible if $k=3$ and $\deg_{V\setminus T}(v)=2$.

Finally, for two degree-2 vertices $u,v$, each having unique neighbours $x_u,x_v$ in $T\setminus U$, there cannot be a path of length at most 4 between $u,v$ in $G[U]$. Indeed, all the vertices of this path would be in $B_{\mr{bias}}$, which would imply a path of length at most $6$ between $x_u,x_v\in T\setminus T_{\mr{bad}}$ (or a cycle of length at most $6$ involving $x_u=x_v$) in $G[T\cup B_{\mr{bias}}]$.
\end{proof}

For (4), we proceed in essentially the same way as \cref{lem:rotate-bulk,lem:rotate-core}.

\begin{proof}[Proof of \cref{lem:extraction}(4)]
Let $G[T',E]$ be the bipartite graph of edges between $T'$ and $E$. Let $H_1,H_2$ be bipartite graphs with the same bipartition $X\cup Y$, such that every vertex in $X$ has the same degree in $H_1$ as it does in $H_2$. Then, for any outcome of $G$ such that $G[T',E]=H_1$, we can swap $G[T',E]$ with $H_2$ to obtain an outcome of $G$ such that $G[T',E]=H_2$. So, $H_1$ and $H_2$ are equally likely to occur as $G[T',E]$. (It is important that this swap can never change the sets $E$ or $T'$, and can never cause the degree of any vertex to drop below $k$.)
\end{proof}

The most involved of the four parts of \cref{lem:extraction} is (3), for which we perform some calculations in the configuration model. The calculations are similar to the ones in \cref{lem:bulk-sequence,lem:BS}, so we will be brief with the details.

\begin{proof}[Proof of \cref{lem:extraction}(3)]
Let $\mbf d=(d_1,\dots,d_n)$ be the degree sequence of $G$, so as in the proof of \cref{lem:bulk-sequence}, by \cref{lem:core-degree-sequence} and a concentration inequality, the following properties hold whp. (Recall \cref{def:CW,def:CW-mod}).
\begin{itemize}
    \item $T=\beta n+o(n)$.
    \item Each $d_v\le\log n$.
    \item For all $t$, there are $\rho_t n+O(n^{3/4})$ vertices $v\in V$ with $d_v=t$.
    \item For all $t\ge k$, there are $(1-\alpha)\rho_t n+O(n^{3/4})$ vertices $v\in V\setminus S$ with $d_v=t$.
    \item For all $t\ge \Delta$, there are $\alpha \rho_t n+O(n^{3/4})$ vertices $v\in T$ with $d_v=t$.
\end{itemize}
Condition on an outcome of $\mbf d$ satisfying the above properties. Now, $G$ is a uniformly random graph with degree sequence $\mbf d$, which we may study with the configuration model. So, consider a uniformly random configuration with degree sequence $\mbf d$, defined in terms of a random permutation $\sigma$. Let $G^*$ be the resulting contracted multigraph. We may interpret the random sets $B_{\mr{bias}},T_{\mr{bad}},E,T_{\mr{low}},T'$ as being defined with respect to $G^*$.

Define $B_{\mr{safe}}$ to be the set of vertices $v$ in $G^*[V\setminus S]$ for which $d_v=k+1$ and $\deg_{V\setminus S}(v)=k$, and for which all vertices $w$ within distance $2$ of $v$ in $G^*[V\setminus S]$ have $d_w=\deg_{V\setminus S}=k$. Note that $E\supseteq B_{\mr{safe}}$. Let $X$ be the number of vertices in $T$ with fewer than $\sqrt{\Delta}$ neighbours in $B_{\mr{safe}}$, and let $Y$ be the number of paths of length at most 6 (in $G^*$) between an ordered pair of vertices $(u,v)\in T^2$ (where we count a cycle of length greater than zero as a path, in the case $u=v$). It suffices to show that $X,Y\le (\beta/16) n$ whp. By a concentration inequality for random permutations, we have $X=\E X+o(n)$ and $Y=\E Y+o(n)$ whp, so it actually suffices to show that $\E X,\E Y\le (\beta/20) n$.

First, we consider $Y$. For any sequence $t_1,\dots,t_{r-1}$ (with $r\le 6$) there are
\[(2m\gamma+O(n^{3/4}\log n))^2\prod_{i=1}^{r-1} t_i(t_i-1)(\rho_{t_i} n+O(n^{3/4}))\]
ways to choose a sequence of distinct points $u,x_1,y_1,\dots,x_{r-1},y_{r-1},v$, such that $u,v$ are each in a bucket corresponding to a vertex in $T$, and for each $i$, the points $x_i,y_i$ are in a common bucket corresponding to a vertex $v$ with $d_v=t_i$. The probability any given sequence of $r$ disjoint pairs are chosen in our random configuration is $(2m-1)^{-1}(2m-3)^{-1}\dots(2m-2r+1)^{-1}=(1+o(1))(2m)^{-r}$. It follows that
\begin{align*}\E Y&\le \sum_{r=1}^6 (2m\gamma+o(n))^2\left(\sum_{t=k}^{\log n}t(t-1)\rho_t n+O(n^{3/4}(\log n)^2)\right)^{r-1} \frac{1+o(1)}{(2m)^r}\\
&\le\gamma^2n\sum_{r=1}^6 2^{2-r}\varepsilon^{-|2-r|}\left(\sum_{t=k}^{\infty}t(t-1)\rho_t\right)^{r-1}+o(n).
\end{align*}
Recalling the interpretation of $\rho_t$ as a probability mass of a truncated Poisson distribution, we have $\sum_{t=k}^{\infty}t(t-1)\rho_t\le \lambda^2/Z(\lambda)$ (in particular, the sum is convergent and does not depend on $\Delta$). Also, recall from the proof of \cref{lem:rank-analysis} that $\gamma^2/\beta$ can be made arbitrarily small by taking sufficiently large $\Delta$ (holding $\varepsilon,\alpha,k$ constant). So for large $\Delta$ we have $\E Y\le (\beta/20) n$, as desired.

It remains to consider $X$. Consider any vertex $v\in T$, and let $b_v$ be a set of $\Delta$ points in the bucket corresponding to $v$. Let $Q_v$ be the set of points in $b_v$ which are matched to a point corresponding to a vertex in $B_{\mr{safe}}$ in our random configuration. Note that $x\in Q_v$ if and only if $x$ is matched with a point corresponding to a degree-$(k+1)$ vertex in $V\setminus S$, and the entire 2-neighbourhood of that vertex consists of vertices in $V\setminus S$ which have degree $k$. There are at least $(1-\alpha)\delta_k n+o(n)$ vertices $v\in V\setminus S$ with $d_v=k$, and at least $(1-\alpha)\delta_{k+1} n+o(n)$ vertices $v\in V\setminus S$ with $d_v=k+1$, and in total our configuration model has $2m\le n/\varepsilon$ points. So, for each $x\in b_v$, we have
\[\Pr(x\in Q_v)\ge (\varepsilon (1-\alpha)\delta_{k+1})(\varepsilon (1-\alpha)\delta_k)^{k^3}+o(1).\]
Let $p=(\varepsilon (1-\alpha)\rho_{k+1})(\varepsilon (1-\alpha)\rho_k)^{k^3}/2$; crucially, this quantity does not depend on $\Delta$. Note that the events $x\in Q_v$, for $x\in b_v$, are very nearly independent (we only need to inspect $O(1)$ points to determine whether $x\in Q_v$, and there are only $\Delta$ points in $b_v$), so $|Q_v|$ stochastically dominates the binomial distribution with parameters $\Delta$ and $p$. By a Chernoff bound, the probability that $|Q_v|\le\sqrt\Delta$ (i.e., that $v$ has fewer than $\sqrt{\Delta}$ neighbours in $B_{\mr{safe}}$) is at most $e^{-p\Delta/8}\le 1/21$ for large $\Delta$. It follows that $\E X\le (\beta/20)n$, as desired.
\end{proof}

\section{Finishing the proof}\label{sec:finishing}
\global\long\def\cc{\mathrm{c}}
In this section we combine all the ingredients collected so far and prove \cref{thm:main}. As part of this proof, we will need the following simple observation.

\begin{proof}[Proof of \cref{thm:main}]
For a small constant $\varepsilon>0$, let $\mathcal E_{\mr{large}}$ be the event that $\core_k(G)$ has at least $\varepsilon n$ vertices, at least $k|\core_k(G)|/2+\varepsilon^2 n$ edges, and at most $n/\varepsilon$ edges (i.e., it has linear size, bounded average degree, and linearly many vertices have degree strictly greater than $k$). It follows from the work of Pittel, Spencer and Wormald~\cite{PSW96} that if $\varepsilon=\varepsilon_k>0$ is a small enough constant, then whp either $\core_k(G)=\emptyset$ or $\mc E_{\mr{large}}$ holds (see for example \cite{JL08} for explicit theorem statements which imply this claim).

For a small constant $\theta'>0$, let $\mathcal E_{\mr{odd}}$ be the event that $\core_k(G)$ has at least $\theta' n$ odd-degree vertices. If we condition on an outcome of the numbers of vertices and edges in $\core_k(G)$ such that $\mathcal E_{\mr{large}}$ holds, then \cref{lem:core-degree-sequence} tells us that for sufficiently small $\theta'>0$, the event $\mathcal E_{\mr{odd}}$ holds whp (considering vertices of degree $k$ or $k+1$, depending on the parity of $k$).

Let $\theta=\theta'/2$, and for a small constant $\eta>0$, let $\mathcal{E}_{\mr{UKP}}$ be the event that the conclusion of \cref{lem:UKP} holds (so if $\eta$ is sufficiently small, $\mathcal{E}_{\mr{UKP}}$ holds whp). For the convenience of the reader we recall that in the statement of \cref{lem:UKP} we defined an induced subgraph of $G$ to be \emph{good} if it has minimum degree at least 2, every pair of degree-2 vertices are at distance at least 4, there are fewer than $(\eta/4) n$ vertices within distance 7 of a degree-2 vertex, and there are at least $\theta n$ odd-degree vertices. The conclusion of \cref{lem:UKP} was that every good induced subgraph of $G$ satisfies the unstructured kernel property $\mr{UKP}(2,\eta/3,\eta)$.

For a small constant $\alpha>0$, let $\mathcal E_{\mr{expand}}$ be the event that for every set $S$ of $\alpha n$ vertices, there are at most $\min(\theta'/2,\eta/4)n$ vertices within distance $8$ of a vertex in $S$. If $\alpha$ is sufficiently small, then $\mathcal E_{\mr{expand}}$ holds whp (for example, one can iterate \cref{lem:expansion-estimates-1}).

Let $\mc S$ be the event that $\core_k(G)$ is singular, and fix $\Delta\in \NN$ (which we view as a constant for the purpose of ``whp'' statements and asymptotic notation). It now suffices to show that
\begin{equation}\label{eq:f(Delta)}
\Pr(\mc S\cap \mc E_{\mathrm{UKP}}\cap \mc E_{\mr{odd}}\cap \mc E_{\mr{expand}}\cap\mc{E}_{\mr{large}})\le f(\Delta)+o(1),
\end{equation}
where $f$ is some function satisfying $\lim_{\Delta\to \infty}f(\Delta)=0$. Indeed, recalling that the empty $0\times 0$ matrix is nonsingular, it will follow from the above considerations that $\Pr(\mc S)\le f(\Delta)+o(1)$ and therefore $\limsup_{n\to\infty} \Pr(\mc S)\le f(\Delta)$. Since this is true for all $\Delta$, and $f(\Delta)\to 0$ as $\Delta\to\infty$, the desired result will follow.

Let $V$ be the vertex set of $\core_k(G)$, and let $m$ be the number of edges in $\core_k(G)$. Condition on outcomes of $V,m$ satisfying $\mc E_{\mr{large}}$. By \cref{lem:rotate-core}, $\core_k(G)$ is now a uniformly random graph on the vertex set $V$ with $m$ edges and minimum degree at least $k$. That is to say, $\core_k(G)\sim \mc K(V,m,k)$.

Fix a subset $S\subseteq V$ of $\floor{\vphantom{x^{x^x}}\alpha |V|}$ vertices. As in the statement of \cref{lem:rank}, let $T$ be the set of vertices in $S$ which have at least $\Delta$ neighbours in $V$, and let $T'\subseteq T$ be the set defined in \cref{lem:extraction}. Let $A$ be the adjacency matrix of $G[V\setminus T]$, and let $A'$ be the adjacency matrix of $G[V\setminus T']$. By \cref{lem:rank}, whp $\corank A\le |T|/8$, and by \cref{lem:extraction}(3), whp $|T'|\ge 7|T|/8$. It follows that in fact whp
\begin{equation}\label{eq:initial-rank}
\corank A' \le |T|/8+|T|/8\le |T'|/2.
\end{equation}

With $E$ as defined in \cref{lem:extraction}, condition on all information except the edges between $E$ and $T'$. Moreover condition on the degree of every vertex in $T'$ into $E$. Let $\mc{I}=(V,m,\tilde G,(\deg_E v)_{v\in T'})$ encode the information revealed so far, where $\tilde G$ is the graph consisting of all edges of $G$ not between $T'$ and $E$. Assume that \cref{eq:initial-rank} holds.

By \cref{lem:extraction}(2), every vertex in $T'$ has at least $\sqrt \Delta$ neighbours in $E$, and by \cref{lem:extraction}(4), these neighbours are uniformly random. Now, let $t=|T'|$, and enumerate the vertices in $T'$ as $v_1,\dots,v_{t}$. Let $T_i=\{v_{i+1},\dots,v_{t}\}$, and write $A_i$ for the adjacency matrix of $G[V\setminus T_i]$. Also, write $\mc E^{i}_{\mathrm{UKP}}$ for the event that $A_i$ satisfies $\on{UKP}(2,\eta/3,\eta)$. We also define a sequence of random variables $X_0,\dots,X_{t}$ as follows.
\[X_i=\begin{cases}
\corank A_i&\text{ if }\mc E^0_{\mathrm{UKP}},\dots,\mc E^{i}_{\mathrm{UKP}}\text{ all hold}\\
0&\text{ otherwise.}
\end{cases}
\]
Note that each $X_i$ is determined by $G[V\setminus T_i]$. By \cref{lem:decoupling},
\[\Pr(X_t\ge \max(X_{t-1},1)\,|\,X_1,\dots,X_{t-1},\mc{I})\le\frac{(\log\Delta^{1/2})^C}{\Delta^{1/4}}\]
for an absolute constant $C$. Let $f(\Delta)=11(\log\Delta)^C\Delta^{-1/4}$, so by \cref{eq:initial-rank} and \cref{lem:random-walk}, we have $\Pr(X_t=0\,|\,\mc{I})\ge 1-f(\Delta)$ for sufficiently large $n$. (That $|E|$ is sufficiently large to apply \cref{lem:decoupling} follows from $\mathcal E_{\mr{expand}}$ and the fact that every vertex in $B_{\on{bias}}\cup S$ is within distance $4$ of $S$.)

Now we claim that if $\mc{E}_{\mr{UKP}}\cap\mc{E}_{\mr{odd}}\cap\mc{E}_{\mr{expand}}$ holds then $\corank A_i = X_i$ for all $i$ (so if additionally $X_t=0$ then $\core_k(G)$ is nonsingular). Indeed, for each $G[V\setminus T_i]$, first $\mc{E}_{\mr{odd}}\cap\mc{E}_{\mr{expand}}$ implies that there are at least $\theta'n-(\theta'n/2)=\theta n$ odd-degree vertices, then $\mc{E}_{\mr{expand}}$ implies that there are fewer than $(\eta/4)n$ vertices within distance 7 of a degree-2 vertex (since each degree-2 vertex has a neighbour in $S$ with respect to $G$), and then \cref{lem:extraction}(1) implies that the minimum degree is at least 2 and there is no pair of degree-2 vertices within distance 4 of each other. That is to say, each $G[V\setminus T_i]$ is good, so $\mc{E}_{\mr{UKP}}$ implies that each $\mc E^{i}_{\mathrm{UKP}}$ holds, so $\corank A_i = X_i$ for all $i$, as claimed.

Now we are ready to conclude the proof. For all $\mc{I}$ such that $V,m$ satisfy $\mc{E}_{\mr{large}}$ and such that \cref{eq:initial-rank} holds, we have
\[\Pr(\mc{S}\cap\mc{E}_{\mr{UKP}}\cap\mc{E}_{\mr{odd}}\cap\mc{E}_{\mr{expand}}\,|\,\mc{I})\le\Pr(\mc{S}\cap\corank A_t=X_t\,|\,\mc{I})\le\Pr(X_t = 0\,|\,\mc{I})\le f(\Delta).\]
Using the law of total probability and the fact that \cref{eq:initial-rank} holds (uniformly) whp conditional on any $V,m$ satisfying $\mc{E}_{\mr{large}}$, we find
\[\Pr(\mc{S}\cap\mc{E}_{\mr{UKP}}\cap\mc{E}_{\mr{odd}}\cap\mc{E}_{\mr{expand}}\,|\,V,m)\le f(\Delta)+o(1).\]
Finally, \cref{eq:f(Delta)} follows from the law of total probability applied over pairs $V,m$ satisfying $\mc{E}_{\mr{large}}$. As discussed, this concludes the proof.
\end{proof}

\begin{remark}\label{rem:effective}
The reason that our proof of \cref{thm:main} does not provide an effective probability bound is that we must take $\Delta\to\infty$ sufficiently slowly in terms of the (ineffective) convergence in \cref{thm:BLS}. Apart from the ineffective $f(\Delta)$ probability bound, all relevant events can be shown to hold with probability $1-n^{-\Omega(1)}$.

We now briefly sketch how to modify the proof in order to obtain a probability bound of the form $1-(\log\log n)^{O(1)}/\sqrt{\log n}$. Instead of extracting vertices of degree at least $\Delta$, for some $\Delta$ that is slowly taken to infinity, we separately extract vertices with degree between $\Delta$ and $\log n/(\log\log n)^2$ (call these ``medium-degree vertices''), and vertices with higher degree (call these ``high-degree vertices''), for some large constant $\Delta$. One can show that there are very likely to be at least (say) $n^{4/5}$ high-degree vertices.

Now, let $T'$ be the set of medium and high degree vertices (after being ``cleaned up'' in a similar way to \cref{lem:extraction}). Add these vertices back to $G[V\setminus T']$ one-by-one, starting with the medium-degree vertices. If we do not take $\Delta\to\infty$ then it is no longer true that after adding back the medium-degree vertices we end up with corank zero whp, but if $\Delta$ is a sufficiently large constant it is still possible to show that the corank is very likely at most (say) $n^{1/3}$. (A simple calculation shows that some $\Delta = \Theta_{k}(\lambda^2)$ suffices.) Then, adding back the high-degree vertices boosts the corank all the way to zero, with probability $1-(\log\log n)^{O(1)}/\sqrt{\log n}$.

Of course, this modification does not avoid ineffectivity entirely. Indeed, one can show that after adding the medium-degree vertices, the corank is at most $n^{1/3}$ with probability at least say $1-1/n$, for sufficiently large $n$, but the cutoff for ``sufficiently large'' is an absolute constant depending on the (ineffective) convergence in \cref{thm:BLS}. It may be possible to quantify the methods used by Bordenave, Lelarge and Salez~\cite{BLS11} to prove \cref{thm:BLS}, and thereby obtain completely effective probability bounds, but we have not attempted to do so.
\end{remark}

\bibliographystyle{amsplain_initials_nobysame_nomr.bst}
\bibliography{main.bib}

\end{document}